\documentclass[10pt]{amsart}

\usepackage{amsmath,amsthm,amssymb,amsaddr}
\usepackage{enumerate,comment,bussproofs,tikz,framed}

\newcommand{\class}{\mathsf}
\newcommand{\logic}{\mathcal}
\newcommand{\alg}{\mathbf}
\newcommand{\structure}[1]{\mathbb{#1}}

\newcommand{\set}[2]{\{ #1 \mid #2 \}}
\newcommand{\pair}[2]{\langle #1, #2 \rangle}
\newcommand{\range}[1]{[#1]}
\newcommand{\card}[1]{| #1 |}

\newcommand{\assign}{\mathrel{:=}}

\newcommand{\seq}{\mathrel{\triangleright}}
\newcommand{\dual}{\partial}

\DeclareMathOperator{\Comp}{Comp}
\DeclareMathOperator{\Cons}{Cons}
\DeclareMathOperator{\Class}{Class}
\DeclareMathOperator{\Kalman}{Kalm}

\newcommand{\Fc}{F_{\mathrm{comp}}}
\newcommand{\Eps}{\mathrm{E}}
\newcommand{\FmAlg}{\alg{Fm}}
\newcommand{\Var}{\mathrm{Var}}

\newcommand{\bsubseteq}[1]{\subseteq_{#1}}
\newcommand{\fg}[2]{[#2]_{#1}}
\newcommand{\nonempty}[1]{P_{#1}}
\newcommand{\nonemptydm}[1]{Q_{#1}}

\newcommand{\HsS}{\mathbb{H}_{\mathrm{S}} \mathbb{S}}
\newcommand{\leqhs}{\leq_{\scriptscriptstyle\HsS}}

\newcommand{\True}{\mathsf{t}}
\newcommand{\False}{\mathsf{f}}
\newcommand{\Neither}{\mathsf{n}}
\newcommand{\Both}{\mathsf{b}}

\newcommand{\BA}[1]{\alg{B}_{#1}}

\newcommand{\DM}[1]{\alg{DM}_{#1}}

\newcommand{\FreeDMLat}[1]{\alg{F}_{\mathsf{DML}}(#1)}

\newcommand{\BAm}[1]{\structure{B}_{#1}}
\newcommand{\DMm}[1]{\structure{DM}_{#1}}
\newcommand{\Km}[1]{\structure{K}_{#1}}
\newcommand{\Pm}[1]{\structure{P}_{#1}}

\newcommand{\CL}[1]{\logic{CL}_{#1}}
\newcommand{\CLinfty}{\logic{CL}_{\infty}}
\newcommand{\BD}[1]{\logic{BD}_{#1}}
\newcommand{\BDinfty}{\logic{BD}_{\infty}}
\newcommand{\K}[1]{\logic{K}_{#1}}
\newcommand{\Kinfty}{\logic{K}_{\infty}}
\newcommand{\LP}[1]{\logic{LP}_{#1}}
\newcommand{\LPinfty}{\logic{LP}_{\infty}}
\newcommand{\KO}[1]{\logic{KO}_{#1}}
\newcommand{\KOinfty}{\logic{KO}_{\infty}}

\newcommand{\CLplain}{\logic{CL}}
\newcommand{\BDplain}{\logic{BD}}
\newcommand{\LPplain}{\logic{LP}}
\newcommand{\Kplain}{\logic{K}}

\newcommand{\TRIV}{\logic{TRIV}}
\newcommand{\ATRIV}{\logic{TRIV}_{-}}
\newcommand{\ETL}{\logic{ETL}}
\newcommand{\KOplain}{\logic{KO}}
\newcommand{\Kminus}{\logic{K}_{-}}
\newcommand{\ECQomega}{\logic{ECQ}_{\omega}}

\newcommand{\Amatrix}{\structure{A}_{1}}

\newcommand{\Mfour}{\structure{M}_{4}}
\newcommand{\Mseven}{\structure{M}_{7}}
\newcommand{\Meight}{\structure{M}_{8}}
\newcommand{\Mnine}{\structure{M}_{9}}
\newcommand{\Nseven}{\structure{N}_{7}}
\newcommand{\Neight}{\structure{N}_{8}}
\newcommand{\Nnine}{\structure{N}_{9}}
\newcommand{\Qfour}{\structure{Q}_{4}}
\newcommand{\Qseven}{\structure{Q}_{7}}
\newcommand{\Qeight}{\structure{Q}_{8}}
\newcommand{\Qnine}{\structure{Q}_{9}}

\newtheorem*{theorem*}{Theorem}
\newtheorem{theorem}{Theorem}[section]
\newtheorem{lemma}[theorem]{Lemma}

\newtheorem{fact}[theorem]{Fact}

\author{Adam P\v{r}enosil}
\address{Universit\`{a} degli Studi di Cagliari, Cagliari, Italy}
\email{adam.prenosil@unica.it}
\thanks{The author gratefully acknowledges the support of Fondazione di Sardegna within the project ``Resource sensitive reasoning and logic'', Cagliari, CUP: F72F20000410007. The author is also grateful to Yaroslav Shramko for sharing a draft of the paper~\cite{shramko20} with him, which provided the impetus for the present study of logics of upsets, and for a helpful discussion concerning the appropriate name for the logic of order of the three-element Kleene lattice.}
\title{Logics of upsets of De Morgan lattices}

\begin{document}

\begin{abstract}
  We study logics determined by matrices consisting of a De~Morgan lattice with an upward closed set of designated values, such as the logic of non-falsity preservation in a given finite Boolean algebra and Shramko's logic of non-falsity preservation in the four-element subdirectly irreducible De~Morgan lattice. The key tool in the study of these logics is the lattice-theoretic notion of an $n$-filter. We study the logics of all (complete, consistent, and classical) $n$-filters on De~Morgan lattices, which are non-adjunctive generalizations of the four-valued logic of Belnap and Dunn (of the three-valued logics of Priest and Kleene, and of classical logic). We then show how to find a finite Hilbert-style axiomatization of any logic determined by a finite family of prime upsets of finite De~Morgan lattices and a finite Gentzen-style axiomatization of any logic determined by a finite family of filters on finite De~Morgan lattices. As an application, we axiomatize Shramko's logic of anything but falsehood.
\end{abstract}

  \maketitle

\section{Introduction}
\label{sec: intro}

  Classical propositional logic is the logic of truth preservation in the two-element Boolean algebra: if all the premises are true (i.e.\ take the top value $\True$), then so is the conclusion. Equivalently, it may be defined in terms of preserving non-falsity in the same Boolean algebra: if the conclusion is false (i.e.\ takes the \mbox{bottom} value~$\False$), then so is one of the premises. While these two consequence relations coincide for the two-element Boolean algebra, they come apart when truth values strictly between $\True$ and $\False$ come into play. This is amply illustrated already by the next simplest case: the four-element Boolean algebra. Strict truth and non-falsity in this algebra are indicated by the solid dots in the following diagrams:
\vskip 12pt
\begin{center}
\begin{tikzpicture}[scale=1, dot/.style={inner sep=2.5pt,outer sep=2.5pt}, solid/.style={circle,fill,inner sep=2pt,outer sep=2pt}, empty/.style={circle,draw,inner sep=2pt,outer sep=2pt}]
  \node (0) at (0,0) [empty] {};
  \node (a) at (-1,1) [empty] {};
  \node (b) at (1,1) [empty] {};
  \node (1) at (0,2) [solid] {};
  \draw[-] (0) -- (a) -- (1) -- (b) -- (0);
\end{tikzpicture}
\qquad
\qquad
\begin{tikzpicture}[scale=1, dot/.style={inner sep=2.5pt,outer sep=2.5pt}, solid/.style={circle,fill,inner sep=2pt,outer sep=2pt}, empty/.style={circle,draw,inner sep=2pt,outer sep=2pt}]
  \node (0) at (0,0) [empty] {};
  \node (a) at (-1,1) [solid] {};
  \node (b) at (1,1) [solid] {};
  \node (1) at (0,2) [solid] {};
  \draw[-] (0) -- (a) -- (1) -- (b) -- (0);
\end{tikzpicture}
\end{center}
\vskip 2pt
 Both $x$ and $\neg x$ may be non-false in this algebra, invalidating the principle of ex\-plosion ${x, \neg x \vdash y}$ and the disjunctive syllogism $x, \neg x \vee y \vdash y$. Moreover, both $x$ and $y$ may be non-false without $x \wedge y$ being non-false, invalidating the rule of ad\-junction $x, y \vdash x \wedge y$. On the other hand, classical logic is recovered by restricting to the \mbox{values} $\True$ and $\False$, thus each inference which is valid in this four-valued logic of non-falsity must also be valid in classical logic.

  The above idea of being as permissive as possible about truth, rather than as restrictive as possible, was first introduced in the 1940's by Ja\'{s}kowski~\cite{jaskowski69,jaskowski99}, one of the earliest pioneers of paraconsistent logic. In~his system of \mbox{discussive} logic (also called discursive logic), Ja\'{s}kowski imagines ``the theses \mbox{advanced} by several \mbox{participants} in a discourse [being] combined into a single \mbox{system}'' and drawing logical inferences in a way which allows for different participants to \mbox{advance} conflicting theses without falling into triviality. In~other words, one might \mbox{imagine} taking a non-empty finite set of \mbox{participants}, assigning to each atomic proposition the set of all \mbox{participants} who accept it, and extending this assignment to complex \mbox{propositions}. The discussive consequence relation would then count a formula $\varphi$ as a consequence of a set of formulas $\Gamma$ if whenever each $\gamma \in \Gamma$ is accepted by \emph{some} participant, then so is $\varphi$. By contrast, the \mbox{consequence} relation of classical logic counts $\varphi$ as a consequence of $\Gamma$ if whenever each $\gamma \in \Gamma$ is accepted by \emph{every} participant, then so is $\varphi$. These two approaches, of course, amount to studying the preservation of non-falsity and truth on finite Boolean algebras (or some related family of algebras, depending on the precise signature).\footnote{Ja\'{s}kowski's discussive logic, which he calls $D_{2}$ to indicate that it is a two-valued discussive system, in fact goes beyond the signature of Boolean algebras. The characteristic connective of his system is the discussive implication: essentially, $\Diamond x \rightarrow y$ in the language of the modal logic S5. One might also include the discussive conjunction $\Diamond x \wedge y$. Note that our perspective on discussive logic is somewhat anachronistic: Ja\'{s}kowski in fact viewed his system as a set of theorems rather than a consequence relation. However, the above con\-sequence relation directly corresponds to his discussive implication connective via the deduction theorem.}

  The same idea of preserving non-falsity rather than truth was advanced more recently by Shramko~\cite{shramko19abf} in a slightly different context. Instead of the four-element Boolean algebra, Shramko considered the four-element De \mbox{Morgan} lattice $\DM{1}$ which forms the algebraic semantics of the logic of Belnap and Dunn~\cite{dunn76,belnap77a,belnap77b}. The~only difference between the two algebras is in how they interpret negation: negation has two fixpoints in~the four-element De~Morgan lattice, which we call~$\Neither$ (for Neither True nor False) and~$\Both$ (for Both True and False). The top and bottom elements are still denoted $\True$ and $\False$. While Belnap–Dunn logic is defined in terms of preserving truth, which is interpreted in this context by the set $\{ \True, \Both \}$, recently Pietz \& Riveccio~\cite{pietz+rivieccio13} studied the logic defined in terms of preserving exact truth, which is interpreted in this context as the set $\{ \True \}$. Shramko's dual \mbox{proposal} is to consider the logic defined in terms of preserving anything but exact falsehood, i.e.\ preserving the set $\{ \True, \Neither, \Both \}$. The relationship between these last two logics thus mimics the relationship between the logic of truth and the logic of non-falsity over the four-element Boolean algebra, only the negation operator differs.

  These examples serve to illustrate that while non-classical logicians are largely concerned with systems which validate the adjunction rule $x, y \vdash x \wedge y$, there may be natural reasons to relax this rule and consider logics determined by upward closed sets (upsets) rather than lattice filters. In particular, the present paper will be concerned with logics determined by upsets of De~Morgan lattices, thus subsuming the two four-valued logics of non-falsity discussed above.

  The key tool at our disposal is the notion of an $n$-filter on a distributive lattice, developed by the present author in~\cite{prenosil22}. We~already observed that the set of non-zero elements of the four-element Boolean algebra does not form a lattice filter. It~does, however, form the next best thing: a \emph{$2$-filter}. This is an upset $F$ such that
\begin{align*}
  x \wedge y, y \wedge z, z \wedge x \in F & \implies x \wedge y \wedge z \in F.
\end{align*}
  More generally, consider the set $\nonempty{n}$ of non-zero elements of the finite Boolean lattice with $n$ atoms $\BA{n}$. This is an~\emph{$n$-filter}: an upset $F$ such that for each non-empty finite $X \subseteq F$ 
\begin{align*}
  \bigwedge Y \in F \text{ for each } Y \subseteq X \text{ with } 1 \leq \card{Y} \leq n & \implies \bigwedge X \in F.
\end{align*}
  Without loss of generality, we may restrict to $\card{X} = n+1$ in this definition.

  The $n$-filter $\nonempty{n}$ is \emph{prime} in the usual sense: $x \vee y \in \nonempty{n}$ implies that either $x \in \nonempty{n}$ or $y \in \nonempty{n}$. The relationship between $n$-filters on distributive lattices and the prime $n$-filter $\nonempty{n}$ on $\BA{n}$ is entirely analogous to the relationship between filters on distributive lattices and the prime filter $\nonempty{1} = \{ \True \}$ on $\BA{1}$.

\begin{theorem*}
  The $n$-filters on a distributive lattice are precisely the intersections of prime $n$-filters. The prime $n$-filters on a distributive lattice are precisely the homomorphic preimages of the prime $n$-filter $\nonempty{n} \subseteq \BA{n}$.
\end{theorem*}

  The first task that we undertake is to extend this theorem about $n$-filters on distributive lattices to De~Morgan lattices. In this setting, the two-element Boolean algebra $\BA{1}$ is replaced by the four-element subdirectly irreducible De~Morgan lattice $\DM{1}$ described above. The role of the prime filter $\{ \True \}$ on $\BA{1}$ is taken over by the prime filter $\{ \True, \Both \}$ on $\DM{1}$. It remains to find the analogue of the $n$-filter $\nonempty{n}$.

  The structures $\BAm{n} \assign \pair{\BA{n}}{\nonempty{n}}$ are what we call the \emph{dual powers} of the structure $\BAm{1} \assign \pair{\BA{1}}{\{ \True \}}$. The same construction applied to $\DMm{1} \assign \pair{\DM{1}}{\{ \True, \Both \}}$ yields the finite structures $\DMm{n} \assign \pair{\DM{n}}{\nonemptydm{n}}$, where
\begin{align*}
  & \DM{n} \assign (\DM{1})^{n}, & & \langle a_{1}, \dots, a_{n} \rangle \in Q_{n} \iff a_{i} \in \{ \True, \Both \} \text{ for some } 1 \leq i \leq n. 
\end{align*}
  Note the existential rather than universal quantifier in the definition of $\nonemptydm{n}$. Likewise, the dual powers of the three-element substructures $\Pm{1}$ and $\Km{1}$ of $\DMm{1}$ with the universe $\{ \True, \Both, \False \}$ and $\{ \True, \Neither, \False \}$ respctively will be denoted $\Pm{n}$ and $\Km{n}$.

  We shall call an upset $F$ of a De~Morgan lattice \emph{complete} if $x \vee \neg x \in F$ and moreover $x \in F$ implies $x \wedge (y \vee \neg y) \in F$. We call it \emph{consistent} if it is not total ($x \notin F$ for some $x$) and moreover $(x \wedge \neg x) \vee y \in F$ implies $y \in F$. Finally, we call it \emph{classical} if it is both complete and consistent.

  The following theorem now extends the above characterization of $n$-filters on distributive lattices to De~Morgan lattices (see Theorems~\ref{thm: dm preimages} and~\ref{thm: dm intersections}).

\begin{theorem*}
  The (complete, consistent, classical) $n$-filters on a De~Morgan lattice are precisely the intersections of (complete, consistent, classical) prime $n$-filters. The (complete, consistent, classical) prime $n$-filters on a De~Morgan lattice are precisely the homomorphic preimages of the designated set of $\DMm{n}$ ($\Pm{n}$, $\Km{n}$, $\BAm{n}$).
\end{theorem*}

  This immediately yields completeness theorems for the non-adjunctive analogues $\BD{n}$, $\LP{n}$, $\K{n}$, and $\CL{n}$ of the four-valued logic of Belnap and Dunn, the three-valued logics of Priest and Kleene and two-valued classical logic, where the rule of adjunction $x, y \vdash x \wedge y$ is replaced by the $n$-adjunction rule
\begin{align*}
  \set{\bigwedge_{j \neq i} x_{j}}{1 \leq i \leq n+1} \vdash x_{1} \wedge \dots \wedge x_{n+1}.
\end{align*}
  Here $\bigwedge_{j \neq i} x_{j}$ is the submeet of $x_{1} \wedge \dots \wedge x_{n+1}$ which omits $x_{i}$. A more complicated completeness theorem for the non-adjunctive analogues $\KO{n}$ of the logic of order of Kleene lattices can also be obtained in this way.

  The second problem that we consider is the following: given a finite set of \mbox{finite} matrices consisting of a De~Morgan lattice with a designated upset, how do we axiomatize the corresponding logic? In other words, how do we axiomatize a \emph{finitely generated} extension of the logic $\BDinfty$ of all upsets of De~Morgan lattices?

  We show how to do this in two special cases: if all of the upsets are prime, and if all of them are filters. The extensions of $\BDinfty$ which arise in this way may be characterized syntactically by means of certain meta-rules, i.e.\ implications between valid rules. One of these meta-rules is well known: a logic $\logic{L}$ is said to enjoy the \emph{proof by cases property (PCP)} in case
\begin{align*}
  \text{$\Gamma, \varphi_{1} \vee \varphi_{2} \vdash_{\logic{L}} \psi$ if $\Gamma, \varphi_{1} \vdash_{\logic{L}} \psi$ and $\Gamma, \varphi_{2} \vdash_{\logic{L}} \psi$.}
\end{align*}
  We extend this to what we call the \emph{$n$-proof by cases property ($n$-PCP)}. For example, the $2$-PCP states that
\begin{align*}
  \Gamma, \varphi_{1} \vee \varphi_{2} \vee \varphi_{3} \vdash_{\logic{L}} \psi \text{ if }\Gamma, \varphi_{1} \vee \varphi_{2} \vdash_{\logic{L}} \psi \text{ and } \Gamma, \varphi_{2} \vee \varphi_{3} \vdash_{\logic{L}} \psi \text{ and } \Gamma, \varphi_{3} \vee \varphi_{1} \vdash_{\logic{L}} \psi.
\end{align*}
  The Exactly True Logic of Pietz \& Rivieccio~\cite{pietz+rivieccio13} is a natural example of a logic which enjoys the $2$-PCP but not the PCP.

  The following theorems now describe the finitely generated extensions of $\BDinfty$ in the two special cases considered above (see Theorems~\ref{thm: bdinfty with pcp} and~\ref{thm: bd1 with n-pcp}).

\begin{theorem*}
  The following are equivalent for each extension $\logic{L}$ of $\BDinfty$:
\begin{enumerate}[(i)]
\item $\logic{L}$ is a finitary extension of $\BD{n}$ with the PCP,
\item $\logic{L}$ is complete with respect to some set of substructures of $\DMm{n}$,
\item $\logic{L}$ is complete with respect to some finite set of finite structures of the form $\pair{\alg{L}}{F}$ where $\alg{L}$ is a De~Morgan lattice and $F$ is a prime $n$-filter of $\alg{L}$.
\end{enumerate}
  Some such $n$ exists for each finitely generated extension of $\BDinfty$ with the PCP.
\end{theorem*}

\begin{theorem*}
  The following are equivalent for each extension $\logic{L}$ of $\BD{1}$:
\begin{enumerate}[(i)]
\item $\logic{L}$ is a finitary and enjoys the $n$-PCP,
\item $\logic{L}$ is complete with respect to some set of substructures of $(\DMm{1})^{n}$,
\item $\logic{L}$ is complete with respect to some finite set of finite structures of the form $\pair{\alg{L}}{F}$ where $\alg{L}$ is a De~Morgan lattice and $F$ is an $n$-prime upset of $\alg{L}$.
\end{enumerate}
  Moreover, some such $n$ exists for each finitely generated extension of $\BD{1}$.
\end{theorem*}

  These theorems organize the finitely generated extensions of $\BDinfty$ with the PCP and the finitely generated extensions of $\BD{1}$ into two hierarchies, where each level consists of logics which are complete with respect to families of substructures of the finite structure $\DMm{n}$ or $(\DMm{1})^{n}$. This gives us a brute force but terminating algorithm for axiomatizing such extensions: to axiomatize such a logic, it suffices to find the level at which it lies and then separate it from other logics at this level by means of finitely many finitary rules.

  In the first case, we obtain a finite Hilbert-style axiomatization. In the second case, we characterize such a logic as the smallest extension of $\BDinfty$ which enjoys the $n$-PCP and moreover validates some finite set of finitary rules. This in effect yields a Gentzen-style axiomatization. It is worth noting that this provides a finite Gentzen-style axiomatization for some logics which do not have a finite Hilbert-style axiomatization, such as the logic of the matrix $\DMm{1} \times \BAm{1}$ (see~\cite{prenosil21}).

  We explicitly describe the first two levels of these two hierarchies. In particular, we axiomatize Shramko's logic of anything but falsehood~\cite{shramko19abf} (see Theorem~\ref{thm: completeness for abf}).

\begin{theorem*}
  The logic determined by the matrix $\pair{\DM{1}}{\{ \True, \Neither, \Both \}}$ is the extension of~$\BDinfty$ by the $2$-adjunction rule $x \wedge y, y \wedge z, z \wedge x \vdash x \wedge y \wedge z$, the law of the excluded middle $\emptyset \vdash x \vee \neg x$, and the rule $x \vee y, \neg x \vee y \vdash (x \wedge \neg x) \vee y$.
\end{theorem*}

  Finally, extensions of $\BDinfty$ can be classified according to their position in the Leibniz hierarchy of abstract algebraic logic. We show that classical logic is the only protoalgebraic extension of $\BDinfty$ and that $\BDinfty$ has a smallest truth-equational extension (or two minimal truth-equational extensions if we include the top and bottom constants in the signature). Classification in the Frege hierarchy proves to be more complicated.

  Before proceeding to fill in the details of the above narrative, let us remark that the present study can be thought as an extension of the investigation of so-called \emph{super-Belnap logics}, initiated by Rivieccio~\cite{rivieccio12} and further pursued in~\cite{albuquerque+prenosil+rivieccio17,prenosil21}. While super-Belnap logics in the strict sense of the term are defined as the extensions of the four-valued logic of Belnap and Dunn, in our notation $\BD{1}$, part of the appeal of broadening our perspective to cover all extensions of $\BDinfty$ is that doing so restores some measure of symmetry between paraconsistency and para\-completeness. For example, in addition to the family of logics axiomatized by increasingly stronger forms of the principle of explosion
\begin{align*}
  (x_{1} \wedge \neg x_{1}) \vee \dots \vee (x_{n} \wedge \neg x_{n}) \vdash y,
\end{align*}
  one might wish to study the family of logics axiomatized by increasingly stronger forms of the law of the excluded middle
\begin{align*}
  \emptyset \vdash (x_{1} \vee \neg x_{1}) \wedge \dots \wedge (x_{n} \wedge \neg x_{n}).
\end{align*}
  But in the presence of adjunction these rules all collapse into $\emptyset \vdash x \vee \neg x$! Relaxing the rule of adjunction therefore allows us to draw finer distinctions between logical principles, just like relaxing the principle of explosion does. (However, we do not pursue this line of thought in the present paper.)

  The above idea of shifting one's perspective on extensions of Belnap–Dunn logic in order to achieve a greater degree of symmetry among these extensions was first suggested by Shramko~\cite{shramko20}. The present paper can therefore partly be considered as an elaboration on Shramko's proposal. Where our paper differs from Shramko's is in how exactly we achieve this. Shramko in effect discards the adjunction rule by considering \textsc{Fmla}-\textsc{Fmla} consequence relations rather than \textsc{Set}-\textsc{Set} ones, i.e.\ consequence relations of the form $\gamma \vdash \varphi$ rather than $\Gamma \vdash \varphi$. However, one can also achieve the same goal by considering extensions of $\BDinfty$ rather than~$\BD{1}$. In a way, the logic $\BDinfty$ corresponds to the \textsc{Fmla}-\textsc{Fmla} fragment of $\BD{1}$: we have $\Gamma \vdash_{\BDinfty} \varphi$ if and only if ${\gamma \vdash_{\BD{1}} \varphi}$ for some $\gamma \in \Gamma$. Shramko's investigation of the duals of known extensions of $\BD{1}$ can therefore also be pursued within a \textsc{Set}-\textsc{Fmla} framework. This has the advantage of allowing us to rely on the well-developed framework of abstract algebraic logic.

\section{$n$-filters on De~Morgan lattices}

  We start by reviewing some basic results about $n$-filters on distributive lattices proved in~\cite{prenosil22}. Then we extend these results to (complete, consistent, classical, and Kalman) $n$-filters on De~Morgan lattices. As an immediate consequence, we obtain completeness theorems for logics of (complete, consistent, classical, and Kalman) $n$-filters on De~Morgan lattices.

\subsection{$n$-filters on distributive lattices}

  The notation $\range{n} \assign \{ 1, 2, \dots, n \}$ and $X \bsubseteq{n} Y$, meaning that $X \subseteq Y$ and $\card{X} \in \range{n}$, will be used throughout the paper. An \emph{$n$-filter} on a distributive lattice is defined for $n \geq 1$ as an upset $F$ such that for each non-empty finite set $Y \subseteq F$ (without loss of generality with $\card{Y} = n+1$)
\begin{align*}
  \text{if } \bigwedge X \in F \text{ for each } X \bsubseteq{n} Y, \text{ then } \bigwedge Y \in F.
\end{align*}
  A $1$-filter is a lattice filter in the usual sense of the word. Each $m$-filter is an $n$-filter for $m \leq n$, and each upset of a finite distributive lattice is an $n$-filter for some $n$.

  An upset $F$ of $\alg{L}$ will be called \emph{prime} in case $x \vee y \in F$ implies $x \in F$ or $y \in F$.\footnote{Observe that the empty set is prime according to this definition. This is because we wish to characterize prime $n$-filters as precisely the homomorphic preimages of a certain canonical prime $n$-filter. If we restricted to bounded distributive lattices and homomorphisms which preserve the lattice bounds, then it would be appropriate to further require that a prime upset by non-empty. This is in effect what we do in the case of Boolean algebras.} The canonical example of a prime $n$-filter is the set $\nonempty{n}$ of non-zero elements of the finite Boolean lattice over $n$ atoms $\BA{n}$. This is not an $m$-filter for any $m < n$.

  The relationship between prime $n$-filters on distributive lattices and the prime $n$-filter $\nonempty{n}$ extends the relationship between prime filters and the prime filter~$\{ \True \}$.

\begin{theorem}[Prime $n$-filters on distributive lattices]
  The prime $n$-filters on a distributive lattice are precisely the homo\-morphic preimages of $\nonempty{n} \subseteq \BA{n}$.
\end{theorem}

  In order to understand how the set $\nonempty{n} \subseteq \BA{n}$ arises in this context, it will be helpful to review dual products and strict homomorphisms of structures. Throughout the paper, the structures that we consider are \emph{logical matrices}: they consist of an algebra $\alg{A}$ and a set $F \subseteq \alg{A}$, called the \emph{designated set} of the structure. The direct product of a family of structures $\pair{\alg{A}_{i}}{F_{i}}$ with $i \in I$ is the structure
\begin{align*}
  \prod_{i \in I} \pair{\alg{A}_{i}}{F} \assign \Big \langle {\prod_{i \in I} \alg{A}_{i}}, {\bigcap_{i \in I} \pi_{i}^{-1}[F_{i}]} \Big \rangle,
\end{align*}
  where $\pi_{i}\colon \alg{A} \to \alg{A}_{i}$ are the projection maps from $\alg{A} \assign \prod_{i \in I} \alg{A}_{i}$. The \emph{dual product} of the family of structures $\pair{\alg{A}_{i}}{F_{i}}$ with $i \in I$, on the other hand, is the structure
\begin{align*}
  \bigotimes_{i \in I} \pair{\alg{A}_{i}}{F_{i}} \assign \Big \langle {\prod_{i \in } \alg{A}_{i}}, {\bigcup_{i \in I} \pi_{i}^{-1}[F_{i}]} \Big \rangle.
\end{align*}
  In particular, observe that $\BAm{n} \assign \pair{\BA{n}}{\nonempty{n}}$ is the $n$-th dual power of the structure $\BAm{1} \assign \pair{\BA{1}}{\nonempty{1}} = \pair{\BA{1}}{\{ \True \}}$. We shall use the notation $\BAm{1}^{\otimes n}$ to indicate dual powers.

  The correspondence between prime filters on a distributive lattice $\alg{L}$ and homo\-morphisms into $\BA{1}$ is well known: each prime filter $F$ on $\alg{L}$ corresponds to a unique homomorphism $h\colon \alg{L} \to \BA{1}$ such that $F = h^{-1} \{ \True \}$. Let us phrase this in terms of structures. A \emph{strict homomorphism} of structures $h\colon \pair{\alg{A}}{F} \to \pair{\alg{B}}{B}$ is a homomorphism of algebras $h\colon \alg{A} \to \alg{B}$ such that $F = h^{-1}[G]$.

\begin{lemma}[Homomorphism lemma for $\BAm{1}$]
  Let $F$ be a prime \mbox{filter} on a distributive lattice $\alg{L}$. Then there is a strict homomorphism ${h_{F}\colon \pair{\alg{L}}{F} \to \BAm{1}}$, namely
\begin{align*}
  h_{F}(x) & \assign \True \text{ if } x \in F, & h_{F}(x) & \assign \False \text{ if } x \notin F.
\end{align*}
\end{lemma}

  A family of strict homomorphisms $h_{i}\colon \pair{\alg{A}}{F_{i}} \to \pair{\alg{B}_{i}}{G_{i}}$ with $i \in I$ may be packaged into a single strict homo\-morphism of structures in two different ways:
\begin{align*}
  h_{\scriptscriptstyle\cap}\colon \pair{\alg{A}}{\bigcap_{i \in I} F_{i}} \to \prod_{i \in I} \pair{\alg{B}_{i}}{G_{i}}, & & h_{\scriptscriptstyle\cup}\colon \pair{\alg{A}}{\bigcup_{i \in I} F_{i}} \to \bigotimes_{i \in I} \pair{\alg{B}_{i}}{G_{i}}.
\end{align*}
  A finite family of prime filters $F_{i}$ on $\alg{L}$ for $i \in \range{n}$ thus yields a strict homo\-morphism $h\colon \pair{\alg{L}}{F} \to \BAm{n}$ for $F \assign F_{1} \cup \dots \cup F_{n}$. The characterization of prime $n$-filters as the homomorphic preimages of the canonical prime $n$-filter $\nonempty{n}$ on $\BA{n}$ is now an immediate consequence of the following fact. (Note that the analogous claim for arbitrary $n$-filters is far from true.)

\begin{fact}
  Prime $n$-filters on a distributive lattice are precisely the unions of families of at most $n$ prime filters.
\end{fact}

  The notion of an $n$-prime filter on a distributive lattice $\alg{L}$ is dual to that of a prime $n$-filter: a filter $F$ is \emph{$n$-prime} if its complement is an $n$-ideal (i.e.\ an $n$-filter in the order dual lattice). For example, a filter $F$ is $2$-prime if
\begin{align*}
  x \vee y \vee z \in F & \implies x \vee y \in F \text{ or } y \vee z \in F \text{ or } z \vee y \in F.
\end{align*}
  Because the $n$-prime filters on $\alg{L}$ are precisely the complements of prime $n$-filters on the order dual of $\alg{L}$, the characterization of prime $n$-filters as the homomorphic preimages of $\nonempty{n}$ yields a dual characterization of $n$-prime filters.\footnote{The reader may be tempted at this point to put the two definitions together and consider $m$-prime $n$-filters. Let us therefore pre-emptively warn the reader that in general it is \emph{not} appropriate to define $m$-prime $n$-filters as $n$-filters whose complement is an $m$-ideal. Rather, $m$-prime $n$-filters should be defined as the $m$-prime elements of the lattice of $n$-filters. These definitions coincide if $m=1$ or $n=1$, but a $2$-filter whose complement is a $2$-ideal need not be a $2$-prime $2$-filter.}

  This follows from the De~Morgan duality between direct and dual products:
\begin{align*}
  \overline{\pair{\alg{A}}{F} \otimes \pair{\alg{B}}{G}} = \overline{\pair{\alg{A}}{F}} \times \overline{\pair{\alg{B}}{G}}, \text{ where } \overline{\pair{\alg{A}}{F}} \assign \pair{\alg{A}}{\alg{A} \setminus F}.
\end{align*}
\begin{theorem}[$n$-prime filters on distributive lattices]
  The $n$-prime filters on a distributive lattice are precisely the homomorphic preimages of the $n$-prime filter $\{ \True \}$ on $\BA{n}$.
\end{theorem}

  The relationship between $n$-filters and prime $n$-filters on distributive lattices extends the relationship between filters and prime filters.

\begin{theorem}[$n$-filters on distributive lattices]
  The $n$-filters on a distributive lattice are precisely the intersections of prime $n$-filters.
\end{theorem}

  The proof of this theorem relies on understanding how $n$-filters are generated. Let $U$ be a subset of a distributive lattice $\alg{L}$. Because an arbitrary intersection of $n$-filters is an $n$-filter, there is a smallest $n$-filter which extends $U$. We call this the $n$-filter \emph{generated} by $U$ and denote it $\fg{n}{U}$. It suffices to describe $U$ in case $U$ is a non-empty upset, since $\fg{n}{\emptyset} = \emptyset$ and the $n$-filter generated by $U$ coincides with the $n$-filter generated by its upward closure.

\begin{lemma}[Generating $n$-filters]
  Let $U$ be a non-empty upset of a distributive lattice $\alg{L}$. Then
\begin{align*}
  a \in \fg{n}{U} \iff  & \text{there is some non-empty finite $X \subseteq U$ such that $\bigwedge X \leq a$} \\
  & \text{and $\bigwedge Y \in U$ for each non-empty $Y \subseteq X$ with $\card{Y} \leq n$.}
\end{align*}
\end{lemma}

  The following lemma, where we use the notation $\fg{n}{U, x} \assign \fg{n}{U \cup \{ x \}}$, is a consequence of this description of $\fg{n}{U}$.

\begin{lemma} \label{lemma: fg cap}
  $\fg{n}{U, x} \cap \fg{n}{U, y} = \fg{n}{U, x \vee y}$.
\end{lemma}

  The standard argument involving a maximal $n$-filter not containing some given element now immediately yields the above characterization of $n$-filters. In fact, $n$-filters may be separated from arbitrary ideals by prime $n$-filters.

\begin{theorem}[Prime $n$-filter separation]
  Let $F$ be an $n$-filter on a distributive lattice which is disjoint from an ideal $I$. Then $F$ extends to a prime $n$-filter which is disjoint from $I$.
\end{theorem}

\subsection{Prime $n$-filters on De~Morgan lattices as homomorphic preimages}

  The above characterization of $n$-filters and prime $n$-filters on distributive lattices extends to \emph{De~Morgan lattices}. These are distributive lattices equipped with an antitone involution $x \mapsto \neg x$, i.e.\ a unary operation satisfying the De~Morgan laws:
\begin{align*}
  \neg \neg x & = x, & \neg (x \vee y) & = \neg x \wedge \neg y, & \neg (x \wedge y) & = \neg x \vee \neg y.
\end{align*}  
  This will involve replacing the prime filter $\{ \True \}$ on $\BA{1}$ by the prime filter $\{ \True, \Both \}$ on the four-element subdirectly irreducible De~Morgan lattice $\DM{1}$ shown in Figure~\ref{fig: dm four}, i.e.\ replacing the structure $\BAm{1} \assign \langle \pair{\BA{1}}{\{ \True \}}$ by $\DMm{1} \assign \pair{\DM{1}}{\{ \True, \Both \}}$. The De~Morgan lattice $\DM{1}$ only differs from the Boolean algebra $\BA{1} \times \BA{1}$ in the interpretation of negation: in $\DM{1}$ we have $\neg \Neither = \Neither$ and $\neg \Both = \Both$.

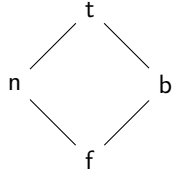
\begin{figure}
\caption{The algebra $\DM{1}$}
\label{fig: dm four}
\vskip 10pt
\begin{center}
\begin{tikzpicture}[scale=1, dot/.style={inner sep=3.5pt,outer sep=3.5pt}, solid/.style={circle,fill,inner sep=2.5pt,outer sep=2.5pt}]
  \node (DM4a) at (0,-1) {$\False$};
  \node (DM4b) at (-1,0) {$\Neither$};
  \node (DM4c) at (1,0) {$\Both$};
  \node (DM4d) at (0, 1) {$\True$};
  \draw[-] (DM4a) edge (DM4b);
  \draw[-] (DM4a) edge (DM4c);
  \draw[-] (DM4b) edge (DM4d);
  \draw[-] (DM4c) edge (DM4d);
\end{tikzpicture}
\end{center}
\end{figure}

  Unlike $\BAm{1}$, the structure $\DMm{1}$ has proper substructures. Besides the two singleton substructures with universes $\{ \Neither \}$ and $\{ \Both \}$, it has the two three-element substructures $\Pm{1}$ and $\Km{1}$ (for Priest and Kleene) with universes $\{ \True, \Both, \False \}$ and $\{ \True, \Neither, \False \}$ respectively, and the two-element substructure $\BAm{1}$ with the universe $\{ \True, \False \}$. The algebraic reducts of $\Pm{1}$ and $\Km{1}$ are \emph{Kleene lattices}: they satisfy $x \wedge \neg x \leq y \vee \neg y$.

  Corresponding to these three substructures, we define three types of upsets on De~Morgan lattices. An upset $F$ of a De~Morgan lattice is \emph{almost complete} if
\begin{align*}
  x \in F & \implies x \wedge (y \vee \neg y) \in F.
\end{align*}
  It is \emph{complete} if it is almost complete and non-empty, or equivalently if it is almost complete and $x \vee \neg x \in F$ for each $x$. It is \emph{almost consistent} if
\begin{align*}
  (x \wedge \neg x) \vee y \in F & \implies y \in F.
\end{align*}
  It is \emph{consistent} if it is almost consistent and not total, i.e.\ $x \notin F$ for some $x$. An \emph{(almost) classical} upset is both complete and (almost) consistent.\footnote{The asymmetry here arises from the fact that one can always talk about the complete upset generated by a given set, but not necessarily about the consistent upset generated by a given set.}

  The homomorphism lemma for $\BAm{1}$ now has an analogue for $\DMm{1}$ and its substructures~\cite[Proposition~3.2]{pynko99}.

\begin{lemma}[Homomorphism lemma for $\DMm{1}$]
  Let $F$ be a prime \mbox{filter} on a De~Morgan lattice $\alg{L}$. Then there is a strict homomorphism ${h_{F}\colon \pair{\alg{L}}{F} \to \DMm{1}}$, namely
\begin{align*}
  h_{F}(x) & \assign \True \text{ if } x \in F, \neg x \notin F, &
  h_{F}(x) & \assign \Neither \text{ if } x \notin F, \neg x \notin F, \\
  h_{F}(x) & \assign \Both \text{ if } x \in F, \neg x \in F, &
  h_{F}(x) & \assign \False \text{ if } x \notin F, \neg x \in F.
\end{align*}
  If $F$ is a complete (consistent, classical) prime filter, then $h_{F}$ is a strict homo\-morphism into~$\Pm{1}$ (into $\Km{1}$, into $\BAm{1}$).
\end{lemma}

  The dual powers of the structure $\DMm{1}$ will be denoted $\DMm{n}$:
\begin{align*}
  \DMm{n} = \pair{\DM{n}}{\nonemptydm{n}} \assign (\DMm{1})^{\otimes n} = \pair{\DM{1}}{\{ \True, \Both \}}^{\otimes n}.
\end{align*}
  Similarly, the dual powers of $\Pm{1}$ and $\Km{1}$ will be denoted $\Pm{n}$ and~$\Km{n}$. Thus
\begin{align*}
  \BAm{n} & \assign (\BAm{1})^{\otimes n}, & \Pm{n} & \assign (\Pm{1})^{\otimes n}, & \Km{n} & \assign (\Km{1})^{\otimes n}, & \DMm{n} & \assign (\DMm{1})^{\otimes n}.
\end{align*}

  Repeating the argument for distributive lattices shows that the prime $n$-filters on De~Morgan lattices are precisely the homomorphic preimages of the designated set of $\DMm{n}$. Some more work is needed, however, to obtain an analogous characterization of complete, consistent, and classical $n$-filters.

\begin{fact} \label{fact: unions}
  The (complete, consistent, classical) prime $n$-filters on a De~Morgan lattice are precisely the unions of non-empty families of at most $n$ (complete, consistent, classical) prime filters.
\end{fact}

\begin{proof}
  The union of any non-empty family of complete (consistent, classical) upsets is complete (consistent, classical), since the implications defining these conditions have at most one premise. Conversely, let $F$ be a non-empty prime $n$-filter on a De~Morgan lattice $\alg{L}$. We know that $F$ is an irredundant union of at most $n$ prime filters on $\alg{L}$, say $F = F_{1} \cup \dots \cup F_{k}$ for $k \leq n$. If $F$ is complete, consider therefore some $a_{i} \in F_{i}$. Because the union is non-redundant, there is some $c_{i} \in F_{i}$ with $c_{i} \notin F_{j}$ for $j \neq i$. Then $a_{i} \wedge c_{i} \in F_{i}$. Because $F$ is complete, $a_{i} \wedge c_{i} \wedge (b \vee \neg b) \in F$. It follows that $a_{i} \wedge c_{i} \wedge (b \vee \neg b) \in F_{i}$ and $a_{i} \wedge (b \vee \neg b) \in F_{i}$. Similarly, if $F$ is consistent and $(x \wedge \neg x) \vee y \in F_{i}$, then $(x \wedge \neg x) \vee (y \wedge c_{i}) \in F_{i}$ for the same $c_{i}$, hence $y \wedge c_{i} \in F$. But $c_{i} \notin F_{j}$ for $j \neq i$, hence $y \wedge c_{i} \in F_{i}$ and $y \in F_{i}$.  
\end{proof}

\begin{theorem}[Complete, consistent, and classical prime $n$-filters] \label{thm: dm preimages}
  The (complete, consistent, classical) prime $n$-filters are precisely the homomorphic preimages of the designated set of $\DMm{n}$ ($\Pm{n}$, $\Km{n}$, $\BAm{n}$).
\end{theorem}

  One further family of upsets merits our attention. We call an upset $F$ of a De~Morgan lattice a \emph{Kalman} upset if
\begin{align*}
  ((x \wedge \neg x) \wedge z) \vee u \in F & \implies ((y \vee \neg y) \wedge z) \vee u \in F,
\end{align*}
  or equivalently
\begin{align*}
  ((x \wedge \neg x) \vee z) \wedge u \in F & \implies ((y \vee \neg y) \vee z) \wedge u \in F.
\end{align*}
  In particular, each almost complete upset is Kalman, as is each almost consistent upset. We use this name to emphasize that these upsets are related to the logic of order of Kleene lattices, which, adding to the multiplicity of names under which it has been known, we call \emph{Kalman's logic of order} here.\footnote{This logic was called Kalman implication by Makinson~\cite{makinson73}, the Kalman consequence system by Dunn~\cite{dunn99}, and Kleene's logic of order by Rivieccio~\cite{rivieccio12}. While there is some logic behind calling filters associated with Kleene lattices Kleene filters, we choose to use Kalman's name here to avoid the ambiguity stemming from the fact that Kleene's name is already attached to several logics in the vicinity. It was Kalman~\cite{kalman58} who, to the best of our knowledge, first axiomatized the variety of what is now called Kleene lattices by what is in our notation the axiom $x \wedge \neg x \leq y \vee \neg y$. One could thus very well also call these Kalman lattices.\label{foot: kalman}} We now explain the purpose of this definition.

\begin{lemma} \label{lemma: kleene theta}
  The smallest congruence $\theta$ on a De~Morgan lattice $\alg{L}$ such that $\alg{L} / \theta$ is a Kleene lattice is the congruence $\theta$ such that $\pair{a}{b} \in \theta$ if and only if
\begin{align*}
  & a \wedge f = b \wedge f \text{ and } \neg f \vee a = \neg f \vee b \text{ for some } f \in \Fc,
\end{align*}
  where $\Fc$ is the filter generated by elements of the form $x \vee \neg x$.
\end{lemma}

\begin{fact}
  An upset ($n$-filter) of a De~Morgan lattice is Kalman if and only if it is the homomorphic preimage of an upset ($n$-filter) of a Kleene lattice.
\end{fact}

\begin{proof}
  Each upset of a Kleene lattice is a Kalman upset, hence so are its homomorphic preimages. Conversely, let $F$ be an order-Kleene upset ($n$-filter) of a De~Morgan lattice $\alg{L}$. Let $\theta$ be the congruence given by the above lemma and let $\pi$ be the projection map $\pi\colon \alg{L} \to \alg{L} / \theta$. If we can show that $a \in F$ and $\pair{a}{b} \in \theta$ implies $b \in F$, then $\pi[F]$ is an upset ($n$-filter) of the Kleene lattice $\alg{L} / \theta$ and $F$ is its preimage via~$\pi$. Thus, suppose that $a \in F$ and $a \wedge f \leq b$ and $a \leq \neg f \vee b$ for some $f \in \Fc$. Then $a = a \wedge (\neg f \vee b) \in F$ and by the Kalman condition $a \wedge (f \vee b) \in F$, hence $a \wedge (f \vee b) \leq (a \wedge f) \vee b = b \in F$.
\end{proof}

  Classical upsets ($n$-filters) admit an analogous characterization.

\begin{lemma} \label{lemma: boolean theta}
  The smallest congruence $\theta$ on a De~Morgan lattice $\alg{L}$ such that $\alg{L} / \theta$ is a Boolean algebra is the congruence $\theta$ such that $\pair{a}{b} \in \theta$ if and only if
\begin{align*}
  & (\neg f \vee a) \wedge f = (\neg f \vee b) \wedge f \text{ for some } f \in \Fc,
\end{align*}
  where $\Fc$ is the filter generated by elements of the form $x \vee \neg x$.
\end{lemma}

\begin{fact}
  An upset ($n$-filter) of a De~Morgan lattice is classical if and only if it is the homomorphic preimage of a non-empty upset ($n$-filter) of a Boolean lattice.
\end{fact}

\begin{proof}
  Each  non-empty upset of a Boolean algebra is classical, hence so are its homomorphic preimages. Conversely, let $F$ be a classical upset ($n$-filter) of a De Morgan lattice $\alg{L}$. Let $\theta$ be the congruence given by the above lemma and let $\pi$ be the projection map $\pi\colon \alg{L} \to \alg{L} / \theta$. If we can show that $a \in F$ and $\pair{a}{b} \in \theta$ implies $b \in F$, then $\pi[F]$ is an upset ($n$-filter) of the Boolean algebra $\alg{L} / \theta$ and $F$ is its preimage via $\pi$. Thus, suppose that $a \in F$ and $(\neg f \vee a) \wedge f \leq \neg f \vee b$. Then $\neg f \vee a \in F$ because $F$ is an upset and $(\neg f \vee a) \wedge f \in F$ because it is complete, hene $\neg f \vee b \in F$ and finally $b \in F$ because $F$ is consistent.
\end{proof}

  Unlike complete, consistent, and classical prime $n$-filters, which are described as the homomorphic preimages of a single prime $n$-filter, Kalman prime $n$-filters will be identified as the homomorphic preimages of a finite family of prime $n$-filters.

\begin{fact}
  The Kalman prime $n$-filters are precisely the unions of at most $n$ Kalman prime filters.
\end{fact}

\begin{proof}
  Each Kalman prime $n$-filter is a homorphic image of a prime $n$-filter on a Kleene lattice, which is a union of at most $n$ prime filters. Their homomorphic preimages are the required Kalman prime filters.
\end{proof}

\begin{fact}
  Each Kalman prime filter is either complete or consistent.
\end{fact}

\begin{proof}
  For prime filter $F$ on a Kleene lattice $\alg{L}$ there is a strict homomorphism $h\colon \pair{\alg{L}}{F} \to \DMm{1}$. Because $\alg{L}$ is Kleene, so is the image of $h$. Thus we either have a strict homomorphism $h\colon \pair{\alg{L}}{F} \to \Pm{1}$ or a strict homomorphism $h\colon \pair{\alg{L}}{F} \to \Km{1}$. In the former case, $F$ is complete. In the latter case, it is consistent. The same therefore holds for each homomorphic preimage of $F$.
\end{proof}

\begin{theorem}[Kalman prime $n$-filters]
  The Kalman prime $n$-filters on a De Morgan lattice are precisely the homomorphic preimages of the designated sets of structures of the form $\Pm{i} \otimes \Km{j}$ for $i + j = n$, where $0 \leq i \leq n$ and $0 \leq j \leq n$.
\end{theorem}

\begin{proof}
  This follows immediately from the previous two facts, taking into account that the complete (consistent) prime filters are precisely the homomorphic pre\-images of the designated set of $\Pm{1}$ ($\Km{1}$).
\end{proof}

  The above results may be dualized in exactly the same way as the corresponding results about distributive lattices. Given an upset $F$ of a De~Morgan lattice $\alg{L}$, by its \emph{De~Morgan dual} we mean its complement $L \setminus F$ as an upset of the order-dual $\alg{L}^{\dual}$ of $\alg{L}$. By extension, the De~Morgan dual of the structure $\pair{\alg{L}}{F}$ is the structure $\dual \pair{\alg{L}}{F} \assign \pair{\alg{L}^{\dual}}{L \setminus F}$. Notice that $\dual \dual \pair{\alg{L}}{F} = \pair{\alg{L}}{F}$.

\begin{fact}
  The De~Morgan duals of complete (consistent, classical, Kalman) prime $n$-filters are precisely the consistent (complete, classical, Kalman) $n$-prime filters.
\end{fact}

  The following theorem is now an immediate consequence of the characterization of prime $n$-filters, given that each strict homomorphism $h\colon \pair{\alg{A}}{F} \to \pair{\alg{B}}{G}$ can also be viewed a strict homomorphism $h\colon \dual \pair{\alg{A}}{F} \to \dual \pair{\alg{B}}{G}$.

\begin{theorem}[Complete, consistent, and classical $n$-prime filters]
  The (complete, consistent, classical) $n$-prime filters are precisely the homomorphic preimages of the designated set of $(\DMm{1})^{n}$ ($(\Pm{1})^{n}$, $(\Km{1})^{n}$, $(\BAm{1})^{n}$).
\end{theorem}

\begin{theorem}[Kalman $n$-prime filters]
  The Kalman $n$-prime filters on a De Morgan lattice are precisely the homomorphic preimages of the designated sets of structures of the form $(\Pm{1})^{i} \times (\Km{1})^{j}$ for $i + j = n$, where $0 \leq i, j \leq n$.
\end{theorem}

\subsection{Prime $n$-filter separation}

  Having described the prime $n$-filters of various kinds as the homomorphic preimage of certain upsets, we now describe arbitrary $n$-filters of these kinds as intersections of prime ones.

  Let $U$ be an upset of a De~Morgan lattice $\alg{L}$. Recall that $\fg{n}{U}$ denotes the $n$-filter generated by $U$. Let $\Fc$ be the filter generated by the elements of the form $x \vee \neg x$. We introduce the upsets $\Comp U$, $\Cons U$, $\Class U$, and $\Kalman U$ for non-empty $U$ as follows: $x \in \Comp U$ if and only if
\begin{align*}
  a \wedge f \leq x \text{ for some } a \in U \text{ and } f \in \Fc,
\end{align*}
  $x \in \Cons U$ if and only if
\begin{align*}
  \neg f \vee x \in U \text{ for some } f \in \Fc,
\end{align*}
  $x \in \Class U$ if and only if
\begin{align*}
  a \wedge f \leq \neg f \vee x \text{ for some } a \in U \text{ and } f \in \Fc,
\end{align*}
  and $x \in \Kalman U$ if and only if
\begin{align*}
  a \wedge f \leq x \text{ and } a \leq \neg f \vee x \text{ for some } a \in U \text{ and } f \in \Fc.
\end{align*}
  We define $\Cons \emptyset \assign \emptyset$, $\Kalman \emptyset \assign \emptyset$, $\Comp \emptyset \assign \Fc$ and $\Class \emptyset \assign \Cons \Fc$.

  Observe that $\Class U = \Cons \Comp U$. On the other hand, it need not be the case that $\Kalman U = \Comp U \cap \Cons U$ (consider the designated set of $\Pm{1} \otimes \Km{1}$).

\begin{lemma}
  $\Cons \fg{n}{U}$ is an $n$-filter. $\fg{n}{\Comp U}$ is complete.
\end{lemma}

\begin{proof}
  Consider $x_{1}, \dots, x_{n+1}$ such that that $\bigwedge_{j \neq i} x_{j} \in \Cons {\fg{n}{U}}$ for $i \in \range{n+1}$. Then for each $i \in \range{n+1}$ there is $f \in \Fc$ such that $\neg f \vee \bigwedge_{j \neq i} x_{j} \in \fg{n}{U}$. We may choose the same $f$ in each case. Thus $\bigwedge_{j \neq i} (\neg f \vee x_{j}) \in \fg{n}{U}$ for each $i$, hence $\neg f \vee \bigwedge_{i \in I} x_{i} = \bigwedge_{i \in I} (\neg f \vee x_{i}) \in \fg{n}{U}$ and $\bigwedge_{i \in I} x_{i} \in \Cons {\fg{n}{U}}$.

  If $a \in \fg{n}{\Comp U}$, then there is a non-empty finite set $X \subseteq \Comp U$ such that $\bigwedge X \leq a$ and $\bigwedge Y \in \Comp U$ for each $Y \bsubseteq{n} X$. Then $a \wedge f \in \fg{n}{\Comp U}$ is witnessed by the set $Z \assign \set{x \wedge f}{x \in X \text{ and } f \in \Fc} \subseteq \Comp U$.
\end{proof}

\begin{theorem}[Generating $n$-filters on De~Morgan lattices]
  Let $U$ be a non-empty upset of a De~Morgan lattice. Then:
\begin{enumerate}[(i)]
\item the complete upset generated by $U$ is $\Comp U$,
\item the almost consistent upset generated by $U$ is $\Cons U$,
\item the almost classical upset generated by $U$ is $\Class U$,
\item the Kalman upset generated by $U$ is $\Kalman U$,
\item the complete $n$-filter generated by $U$ is $\fg{n}{\Comp U}$,
\item the almost consistent $n$-filter generated by $U$ is $\Cons {\fg{n}{U}}$,
\item the almost classical $n$-filter generated by $U$ is $\fg{n}{\Class U}$,
\item the Kalman $n$-filter generated by $U$ is $\fg{n}{\Kalman U}$.
\end{enumerate}
\end{theorem}

\begin{proof}
  Cases (i) and (ii) are obvious. Cases (iii) and (vii) follow from the fact that almost classical upsets ($n$-filters) are precisely the homomorphic preimages of non-empty upsets ($n$-filters) of Boolean algebras and the description of the smallest congruence $\theta$ on $\alg{L}$ such that $\alg{L} / \theta$ is a Boolean algebra (Lemma~\ref{lemma: boolean theta}). It suffices to observe that if $h\colon \alg{A} \to \alg{B}$ is a surjective homomorphism of De~Morgan lattices and $U$ is an upset of $\alg{B}$, then $\fg{n}{h^{-1}[U]} = h^{-1}[\fg{n}{U}]$. Similarly, cases (iv) and (viii) follow from the fact that Kalman upsets ($n$-filters) are precisely the homomorphic preimages of upsets ($n$-filters) of Kleene lattices and the description of the smallest congruence $\theta$ on $\alg{L}$ such that $\alg{L} / \theta$ is a Kleene lattice (Lemma~\ref{lemma: kleene theta}). Finally, cases (v) and (vi) follow from the above lemma.
\end{proof}

\begin{theorem}[Prime separation for $n$-filters on De~Morgan lattices] \label{thm: complete prime n filter separation}
  Let $F$ be a (complete, consistent, classical, Kalman) \mbox{$n$-filter} on a De~Morgan lattice which is disjoint from a non-empty ideal $I$. Then $F$ extends to a complete (consistent, classical, Kalman) prime $n$-filter which is disjoint from $I$.
\end{theorem}

\begin{proof}
  The claim for general $n$-filters is a special case of the corresponding theorem for distributive lattices. For complete $n$-filters, let $G$ be the maximal complete $n$-filter disjoint from $I$ which extends $F$. If $x, y \notin G$, then $\fg{n}{\Comp (G, x)}$ and $\fg{n}{\Comp (G, y)}$ intersect $I$, hence there $i \in \fg{n}{\Comp (G, x)} \cap \fg{n}{\Comp (G, y)}$ for some $i \in I$. That is, $i \in \fg{n}{(\Comp G), x \wedge f}$ and $i \in \fg{n}{(\Comp G), y \wedge f }$ for some $f \in \Fc$. But then $i \in \fg{n}{(\Comp G), (x \vee y) \wedge f \}}$ by Lemma~\ref{lemma: fg cap}, so the complete $n$-filter generated by $G \cup \{ x \vee y \}$ intersects $I$ at $i$. Thus $x \vee y \notin G$.

  For consistent $n$-filters, let $G$ be the maximal consistent $n$-filter disjoint from $I$ which extends $F$. If $x, y \notin G$, then $\Cons {\fg{n}{G, x}}$ and $\Cons {\fg{n}{G, y}}$ intersect $I$. Then $i \in \Cons \fg{n}{G, x} \cap \Cons \fg{n}{G, y} = \Cons (\fg{n}{G, x} \cap \fg{n}{G, y}) = \Cons \fg{n}{G, x \vee y}$ for some $i \in I$ by the previous lemma. The consistent $n$-filter generated by $G \cup \{ x \vee y \}$ thus intersects $I$ at $i$, and $x \vee y \notin G$.

  In the Kalman case (the classical case) there is a surjective homomorphism $h\colon \alg{L} \to \alg{K}$ onto a Kleene lattice (Boolean algebra) $\alg{K}$ such that $F = h^{-1}[G]$ for some $n$-filter $G$ of $\alg{K}$. The downset $J \subseteq \alg{K}$ generated by $h[I]$ is an ideal on $\alg{K}$ which is disjoint from $G$: if $x \in J \cap G$, then $x \leq h(i)$ for some $i \in I$, so $h(i) \in G$ and $i \in F = h^{-1}[G]$, contradicting $F \cap I = \emptyset$. It follows that $G$ extends to a prime $n$-filter on $\alg{K}$ disjoint from $J$, and its preimage with respect to $h$ is a Kalman (classical) prime $n$-filter on $\alg{L}$ disjoint from $I$.
\end{proof}

\begin{theorem}[$n$-filters on De~Morgan lattices] \label{thm: dm intersections}
  The (complete, consistent, classical, Kalman) $n$-filters on a De~Morgan lattice are precisely the intersections of non-empty families of (complete, consistent, classical, Kalman) prime $n$-filters.
\end{theorem}

\subsection{Consequence in logics of upsets}

  The above representation of various types on $n$-filters on De~Morgan lattices as the intersections of homomorphic pre\-images of some canonical $n$-filters immediately yields completeness theorems for non-adjunctive generalizations of Belnap--Dunn logic, Kleene's strong three-valued logic, Priest's Logic of Paradox, Kalman's logic of order, and classical logic. Likewise, the above results describing the generation of $n$-filters of various types immediately yield theorems relating the consequence relations of these logics.

  We shall assume in the following that the reader is familiar with the basic notions of abstract algebraic logic: the notion of a (finitary) logic, a (reduced) model of a logic, an extension of a logic, and a (Hilbert-style) axiomatization of a logic. If not, the reader may consult the textbook~\cite{font16} for the relevant definitions.

  The logics considered here will all be axiomatized relative to the base logic $\BDinfty$, defined as the logic determined by all structures of the form $\pair{\alg{A}}{F}$ where $\alg{A}$ is a De~Morgan lattice and $F$ is an upset of $\alg{L}$. A finite Hilbert-style axiomatization for this logic was found by Shramko~\cite{shramko19fde}. More precisely, Shramko was concerned with axiomatizing the \textsc{Fmla}-\textsc{Fmla} fragment of Belnap--Dunn logic, but this is equivalent to axiomatizing $\BDinfty$ in view of Theorem~\ref{thm: consequence in bd}. The choice of $\BDinfty$ as a base logic is appropriate given that extensions of $\BDinfty$ are precisely the logics complete with respect with respect to some class of structures of the above form.

  Let $\LPinfty$ be the extension of $\BDinfty$ by the rules
\begin{align*}
  & \emptyset \vdash x \vee \neg x, & & x \vdash x \wedge (y \vee \neg y),
\end{align*}
  $\Kinfty$ be the extension of $\BDinfty$ by the rule
\begin{align*}
  (x \wedge \neg x) \vee y \vdash y,
\end{align*}
  $\CLinfty$ be the extension of $\BDinfty$ by all three of above rules, and $\KOinfty$ be the extension of $\BDinfty$ by the rule
\begin{align*}
  ((x \wedge \neg x) \wedge z) \vee u \vdash ((y \vee \neg y) \wedge z) \vee u,
\end{align*}
  or equivalently by the rule
\begin{align*}
  ((x \wedge \neg x) \vee z) \wedge u \vdash ((y \vee \neg y) \vee z) \wedge u.
\end{align*}
  Furthermore, let $\BD{n}$ ($\LP{n}$, $\K{n}$, $\KO{n}$, $\CL{n}$) be the extension of $\BDinfty$ ($\LPinfty$, $\Kinfty$, $\KOinfty$, $\CLinfty$) by the $n$-adjunction rule
\begin{align*}
  \set{\bigwedge_{j \neq i} x_{j}}{i \in \range{n+1}} \vdash x_{1} \wedge \dots \wedge x_{n+1},
\end{align*}
  where $\bigwedge_{j \neq i} x_{j}$ is the submeet of $x_{1} \wedge \dots \wedge x_{n+1}$ omitting $x_{i}$.

\begin{theorem}[Completeness for logics of $n$-filters] \label{thm: completeness for n-filters}
  $\BD{n}$ ($\LP{n}$, $\K{n}$, $\CL{n}$) is complete with respect to the structure $\DMm{n}$ ($\Pm{n}$, $\Km{n}$, $\BAm{n}$). $\KO{n}$ is complete with respect to the family of structures of the form $\Pm{i} \otimes \Km{j}$ for $i+j = n$ ($0 \leq i, j \leq n$).
\end{theorem}

\begin{proof}
  Each reduced model of $\BDinfty$ has the form $\pair{\alg{A}}{F}$ where $\alg{A}$ is a De~Morgan lattice and $F$ is an upset of $\alg{A}$ (because each reduced model of a logic generated by structures of the form $\pair{\alg{A}}{F}$ for $\alg{A}$ in some variety $\class{K}$ has the form $\pair{\alg{B}}{G}$ for some $\alg{B} \in \class{K}$). Thus $\BD{n}$ is complete with respect to the class of all structures the form $\pair{\alg{A}}{F}$ where $\alg{A}$ is a De~Morgan lattice and $F$ is an upset which validates $n$-adjunction, i.e.\ an $n$-filter on $\alg{A}$. But then $F$ is an intersection of a family of prime $n$-filters $F_{i}$, and each of these is a homomorphic preimage of $\nonemptydm{n} \subseteq \DM{n}$. Thus $\pair{\alg{A}}{F}$ embeds into some product of strict homomorphic preimages of the structure $\DMm{n}$. It follows that each De~Morgan lattice with an $n$-filter is a model of the logic determined by $\DMm{n}$. Conversely, $\DMm{n}$ is a De~Morgan lattice with an $n$-filter. An analogous argument applies in the other cases.
\end{proof}

  Note that $\KOinfty < \LPinfty \cap \Kinfty$ and $\KO{n} < \LP{n} \cap \K{n}$ for each $n \geq 2$ even though $\KO{1} = \LP{1} \cap \K{1}$. For example, the rule $x, y \wedge \neg y \vdash x \wedge (z \vee \neg z)$ holds in $\LPinfty \cap \Kinfty$, hence in $\LP{n} \cap \K{n}$ for each~$n$, but it fails in $\KO{2}$ because it fails in $\Pm{1} \otimes \Km{1}$.

  We say that $\logic{L}$ is \emph{complete as a finitary logic} with respect to some class of structures $\class{K}$ if each \emph{finitary} rule $\Gamma \vdash \varphi$ is valid in $\logic{L}$ if and only if in holds in each structure in $\class{K}$. Note that a logic can be finitary and complete as a finitary logic with respect to $\class{K}$ without being complete with respect to $\class{K}$.\footnote{This is analogous to the distinction between a quasivariety generated by a class of algebras $\class{K}$ as a quasivariety and a quasivariety being generated by $\class{K}$ as a prevariety.} We shall see momentarily that this is in fact the case for the logics covered by the following theorem.

\begin{theorem}[Completeness for logics of upsets] \label{thm: completeness for upsets}
  $\BDinfty$ ($\LPinfty$, $\Kinfty$, $\KOinfty$, $\CLinfty$) is complete as a finitary logic with respect to the family of structures $\DMm{n}$ ($\Pm{n}$, $\Km{n}$, $\Pm{i} \otimes \Km{j}$, $\BAm{n}$).
\end{theorem}

\begin{proof}
  $\BDinfty$ is complete as a finitary logic with respect to the class of all finitely generated (i.e.\ finite) De~Morgan lattices equipped with an upset. But each such upset is an $n$-filter for some $n$. Thus $\BDinfty$ is the intersection of the logics $\BD{n}$. An analogous argument applies in the other cases.
\end{proof}

  We can describe consequence in $\BDinfty$, $\KOinfty$, and $\CLinfty$ in terms of equational validity in De~Morgan lattices, Kleene lattices, and Boolean algebras.

\begin{theorem}[Consequence in $\BD{n}$, $\BDinfty$ and $\KO{n}$, $\KOinfty$] \label{thm: consequence in bd}
  Let $ \Gamma \vdash \varphi$ be a logical rule in the signature of De~Morgan lattices. Then:
\begin{enumerate}[(i)]
\item $\gamma \vdash_{\BDinfty} \varphi$ if and only if $\gamma \leq \varphi$ holds in all De~Morgan lattices.
\item $\Gamma \vdash_{\BDinfty} \varphi$ if and only if $\gamma \vdash_{\BDinfty} \varphi$ for some $\gamma \in \Gamma$.
\item $\Gamma \vdash_{\BD{n}} \varphi$ if and only if there is some non-empty finite set of terms $\Phi$ such that $\Gamma \vdash_{\BDinfty} \bigwedge \Delta$ for each $\Delta \bsubseteq{n} \Phi$ and $\bigwedge \Phi \vdash_{\BDinfty} \varphi$.
\end{enumerate}
  The same equivalence relates Kleene lattices, $\KOinfty$, and $\KO{n}$.
\end{theorem}

\begin{proof}
  The proof is identical to the proof of the analogous fact for distributive lattices, see~\cite[Fact~4.5]{prenosil22}.
\end{proof}

  For Boolean algebras, the statement of the fact needs to be modified slightly, since $\CLinfty$ is the logic of \emph{non-empty} upsets.

\begin{theorem}[Consequence in $\CL{n}$ and $\CLinfty$]
  Let $\Gamma \vdash \varphi$ be a logical rule in the signature of Boolean algebras. Then:
\begin{enumerate}[(i)]
\item $\emptyset \vdash_{\CLinfty} \varphi$ if and only if $\True \leq \varphi$ holds in all Boolean algebras.
\item $\gamma \vdash_{\CLinfty} \varphi$ if and only if $\gamma \leq \varphi$ holds in all Boolean algebras.
\item $\Gamma \vdash_{\CLinfty} \varphi$ for $\Gamma$ non-empty if and only if $\gamma \vdash_{\CLinfty} \varphi$ for some $\gamma \in \Gamma$.
\item $\Gamma \vdash_{\CL{n}} \varphi$ if and only if there is some non-empty finite set of terms $\Phi$ such that $\Gamma \vdash_{\CLinfty} \bigwedge \Delta$ for each $\Delta \bsubseteq{n} \Phi$ and $\bigwedge \Phi \vdash_{\CLinfty} \varphi$.
\end{enumerate}
\end{theorem}

\begin{proof}
  This is (a special case of) Theorem~4.6 of~\cite{prenosil22}.
\end{proof}

  We can then relate consequence in $\LP{n}$, $\K{n}$, $\KO{n}$, and $\CL{n}$ to consequence in~$\BD{n}$, and likewise for the corresponding logics of upsets. For this purpose, it will be convenient to introduce the abbreviation
\begin{align*}
  \alpha(\psi_{1}, \dots, \psi_{k}) \assign (\psi_{1} \vee \neg \psi_{1}) \wedge \dots \wedge (\psi_{k} \vee \neg \psi_{k}).
\end{align*}
  The following theorems are direct consequences of the corresponding theorems concerning the generation of $n$-filters and upsets of the appropriate types. We only prove the first one explicitly.

\begin{theorem}[Consequence in $\LP{n}$ and $\LPinfty$]
  $\Gamma \vdash_{\LP{n}} \varphi$ if and only if
\begin{align*}
  \alpha(\psi_{1}, \dots, \psi_{k}), \set{\gamma \wedge \alpha(\psi_{1}, \dots, \psi_{k})}{\gamma \in \Gamma} \vdash_{\BD{n}} \varphi \text{ for some } \psi_{1}, \dots, \psi_{k}.
\end{align*}
  The same equivalence relates $\LPinfty$ and $\BDinfty$.
\end{theorem}

\begin{proof}
  The right-to-left direction is trivial. Conversely, suppose that ${\Gamma \vdash_{\LP{n}} \varphi}$. Let $\FreeDMLat{\Var}$ be the free De~Morgan lattice generated by the set of variables $\Var$ and let $\pi\colon \FmAlg(\Var) \to \FreeDMLat{\Var}$ be the unique homomorphism which restricts to the identity on $\Var$. Observe that $\Gamma \vdash_{\BD{n}} \varphi$ if and only if $\pi(\varphi) \in \fg{n}{\pi[\Gamma]}$.

  If $\Gamma \vdash_{\LP{n}} \varphi$, then $\pi(\varphi)$ is in the complete $n$-filter $F$ on $\FreeDMLat{\Var}$ generated by $\pi[\Gamma]$. That is, $\pi(\varphi) \in \fg{n}{\Comp \pi[\Gamma]}$. Because the map $\pi$ is surjective, this yields formulas $\psi_{1}, \dots, \psi_{k}$ such that $\pi(\varphi) \in \fg{n}{\set{\pi(\gamma \wedge \alpha(\psi_{1}, \dots, \psi_{k}))}{\gamma \in \Gamma}}$, hence $\set{\gamma \wedge \alpha(\psi_{1}, \dots, \psi_{k})}{\gamma \in \Gamma} \vdash_{\BD{n}} \varphi$.
\end{proof}

\begin{theorem}[Consequence in $\K{n}$ and $\Kinfty$]
  $\Gamma \vdash_{\K{n}} \varphi$ if and only if
\begin{align*}
  \Gamma \vdash_{\BD{n}} \neg \alpha(\psi_{1}, \dots, \psi_{k}) \vee \varphi \text{ for some } \psi_{1}, \dots, \psi_{k}.
\end{align*}
  The same equivalence relates $\Kinfty$ and~$\BDinfty$.
\end{theorem}

\begin{theorem}[Consequence in $\KOinfty$]
  $\Gamma \vdash_{\KOinfty} \varphi$ if and only if
\begin{align*}
  \alpha(\psi_{1}, \dots, \psi_{k}) \wedge \gamma \vdash_{\BDinfty} \varphi \text{ and } \gamma \vdash_{\BDinfty} \neg \alpha(\psi_{1}, \dots, \psi_{k}) \vee \varphi
\end{align*}
  for some $\gamma \in \Gamma$ and some $\psi_{1}, \dots, \psi_{k}$.
\end{theorem}

\begin{theorem}[Consequence in $\CL{n}$]
  $\Gamma \vdash_{\CL{n}} \varphi$ if and only if
\begin{align*}
  \alpha(\psi_{1}, \dots, \psi_{k}), \set{\gamma \wedge \alpha(\psi_{1}, \dots, \psi_{k})}{\gamma \in \Gamma} \vdash_{\BD{n}} \neg \alpha(\psi_{1}, \dots, \psi_{k}) \vee \varphi
\end{align*}
  for some $\psi_{1}, \dots, \psi_{k}$. The same equivalence relates $\CLinfty$ and $\BDinfty$.
\end{theorem}

  The logic $\CLinfty$ is not complete with respect to the family of structures $\BAm{n}$. As observed in~\cite{prenosil22}, all of the structures $\BAm{n}$ satisfy the following logical rule, which fails in the Boolean algebra $(\BA{1})^{\omega}$ with every non-zero element designated:
\begin{align*}
  \set{(x_{i} \wedge \neg x_{j}) \vee y}{i, j \in \omega \text{ and } i < j} & \vdash y.
\end{align*} 
  Thanks to the connection between $\CL{n}$ and $\BD{n}$, we obtain an infinitary rule which is valid in $\BD{n}$ for each $n$ but not in $\BDinfty$, or indeed in $\CLinfty$. Because $\BD{n} \leq \LP{n}$ and $\LPinfty \leq \CLinfty$, and likewise for the other logics considered above, we can infer that $\LPinfty$ is not complete with respect to the structures $\Pm{n}$, $\Kinfty$ is not complete with respect to the strutures $\Km{n}$, and $\KOinfty$ is not complete with respect to the structures $\Pm{i} \otimes \Km{j}$. In other words, it is indeed necessary to distinguish between completeness \emph{simpliciter} and completeness as a finitary logic in Theorem~\ref{thm: completeness for upsets}.

  Finally, let us remark that the results about logics proved in this section are in fact corollaries of results about what we called filter classes in~\cite{prenosil22}. If we ab\-stract away from limitations imposed by the cardinality of the set of variables and allow for a proper class of variables, a class of structures arises as the class of all models of a logic if and only if it is a \emph{logical class} of structures: a class closed under sub\-structures, products of structures, strict homomorphic images, and strict homomorphic preimages. A \emph{filter class}, on the other hand, is only required to be closed under the first three constructions.

   Since our representation of various types of $n$-filters as intersections of homomorphic preimages of a certain given $n$-filter or family of $n$-filters does not involve taking strict homomorphic images, we in fact showed that e.g.\ the class of all $n$-filters on De~Morgan lattices is generated \emph{as a filter class} by the structure $\DMm{n}$. Syntactically, moving from logical classes to filter classes corresponds to allowing the use of equalities in the premises of logical rules, i.e.\ allowing for rules of the form $\Eps, \Gamma \vdash \varphi$ where $\Eps $ is a set of equations, $\Gamma$ is a set of formulas, and $\varphi$ is a formula. For example, the rule
\begin{align*}
  x \approx \neg x, y \approx \neg y, z \approx \neg z, x \wedge y, x \wedge z \vdash x \wedge y \wedge z
\end{align*}
  holds in a structure $\pair{\alg{L}}{F}$ in case for each $x, y, z \in \alg{L}$
\begin{align*}
  x = \neg x ~ \& ~ y = \neg y ~ \& ~ z = \neg z ~ \& ~ x \wedge y \in F ~ \& ~ x \wedge z \in F \implies x \wedge y \wedge z \in F.
\end{align*}
  In particular, it holds in the structure $\Km{2}$, therefore by the analogue of Theorem~\ref{thm: completeness for n-filters} it must hold in every consistent $2$-filter on a De~Morgan lattice.

  The relationships established in this section between logical rules valid in $\BDinfty$, logical rules valid in its various extensions and equations valid in De~Morgan lattices (or in some cases Kleene lattices or Boolean algebras) then extend to analogous relationships involving rules of this more general type and quasi-equations, replacing the rule $\Gamma \vdash \varphi$ by $\Eps, \Gamma \vdash \varphi$, the rule $\gamma \vdash \varphi$ by $\Eps, \gamma \vdash \varphi$, and the inequality $\gamma \leq \varphi$ by the implication $\Eps \implies \gamma \leq \varphi$, in a manner entirely analogous to \cite[Fact~4.5]{prenosil22}. We refer the interested reader to~\cite{prenosil22} for more details.

\section{Finitely generated extensions of $\BDinfty$}
\label{sec: extensions}

  In the previous section, we obtained completeness theorems for some prominent logics of upsets of De~Morgan lattices. In this section, we shall be concerned with the general problem of axiomatizing logics determined by a finite family of upsets of finite De~Morgan lattices. We show that such an axiomatization can be found in an entirely mechanical matter in two special cases, namely if all of the upsets are filters or if they are all prime, with the caveat that in the former case the axiomatization involves a certain meta-rule (i.e.\ it is a Gentzen-style axiomatization).

\subsection{Two hierarchies of finitely generated extensions}

  The algorithm for axiomatizing such logics rests of the following simple observation. Given a finite family of prime upsets $F_{i}$ of finite De~Morgan lattices $\alg{A}_{i}$ for $i \in I$, there is some $n$ such that each $F_{i}$ is a prime $n$-filter (because if an upset fails to be an $n$-filter, it must contain at least $n+1$ pairwise incomparable elements). Dually, given a finite family of filters $F_{i}$ of finite De~Morgan lattices $\alg{A}_{i}$ for $i \in I$, there is some $n$ such that each $F_{i}$ is an $n$-prime filter. But each prime $n$-filter is a homomorphic preimage of the designated set of $\DMm{n}$, and each $n$-prime filter is a homomorphic preimage of the designated set of $(\DMm{1})^{n}$, i.e.\ there is a strict homomorphism $h\colon \pair{\alg{A}_{i}}{F_{i}} \to \DMm{n}$ or $h\colon \pair{\alg{A}_{i}}{F_{i}} \to (\DMm{1})^{n}$. The structure $\pair{\alg{A}_{i}}{F_{i}}$ thus determines the same logic as the substructure of $\DMm{n}$ or $(\DMm{1})^{n}$ given by the image of $h$.

  The logics determined by a finite family of prime upsets of finite De~Morgan lattices are therefore precisely the logics determined by some family of substructures of the finite structure $\DMm{n}$ for some $n$. Similarly, the logics determined by a finite family of $n$-prime filters on finite De~Morgan lattices are precisely those determined by some family of substructures of the finite structure $(\DMm{1})^{n}$ for some $n$.

  To obtain a satisfactory theorem, it remains to find a syn\-tactic description of the finitely generated logics determined by prime upsets and by $n$-prime filters. An extension $\logic{L}$ of $\BDinfty$ will be said to enjoy the \emph{proof by cases property (PCP)} in case
\begin{align*}
  \text{$\Gamma, \varphi_{1} \vee \varphi_{2} \vdash_{\logic{L}} \psi$ if $\Gamma, \varphi_{1} \vdash_{\logic{L}} \psi$ and $\Gamma, \varphi_{2} \vdash_{\logic{L}} \psi$.}
\end{align*}
  The converse implication always holds, since $\varphi_{1} \vdash_{\logic{L}} \varphi_{1} \vee \varphi_{2}$ and $\varphi_{2} \vdash_{\logic{L}} \varphi_{1} \vee \varphi_{2}$. The proof by cases property was studied in detail (and at a greater level of abstraction) by Czelakowski~\cite[Section~2.5]{czelakowski01} and Cintula \& Noguera~\cite{cintula+noguera13}. More generally, we say that $\logic{L}$ has the \emph{$n$-proof by cases property ($n$-PCP)} in case
\begin{align*}
  \Gamma, \varphi_{1} \vee \dots \vee \varphi_{n+1} \vdash_{\logic{L}} \psi \text{ if }\Gamma, \bigvee_{j \neq i} \! \varphi_{j} \vdash_{\logic{L}} \psi \text{ for each } i \in \range{n+1}.
\end{align*}
  For $n = 1$ this property reduces to the PCP.\footnote{The $n$-PCP looks superficially similar to the PCP, but it is important to keep in mind that the PCP is an intrinsic property of a logic, while the $n$-PCP is depends on our choice of the disjunction connective. That is, if a logic has a disjunction connective which obeys the PCP, it is unique up to logical equivalence (interderivability), but this is not true for the $n$-PCP.}

\begin{fact}
  Let $\logic{L}$ be a finitary extension of $\BDinfty$. Then $\logic{L}$ enjoys the PCP if and only if it is complete with respect to some class of structures of the form $\pair{\alg{L}}{F}$ where $\alg{L}$ is a De Morgan lattice and $F$ is a prime upset of $\alg{L}$.
\end{fact}

\begin{proof}
  The right-to-left implication is straightforward. Conversely, suppose that $\logic{L}$ has the PCP. If a rule fails in $\logic{L}$, then it fails in some structure of the form $\pair{\FmAlg(\Var)}{T}$, where $T$ is a theory of $\logic{L}$. By the finitarity of $\logic{L}$, each $\logic{L}$-theory is an inter\-section of completely meet irreducible $\logic{L}$-theories, therefore the rule fails in some structure of the form $\pair{\FmAlg(\Var)}{U}$ where $U$ is a completely meet irreducible \mbox{$\logic{L}$-theory}. For each such $U$ there is some $\psi$ such that $\varphi \notin U$ if and only if $U, \varphi \vdash \psi$. By~the PCP, $\varphi_{1} \vee \varphi_{2} \in U$ now implies that $\varphi_{1} \in U$ or $\varphi_{2} \in U$. Then there is a strict surjective homomorphism from $\pair{\FmAlg(\Var)}{U}$ onto a structure $\pair{\alg{L}}{F}$ such that $\alg{L}$ is a De Morgan lattice. But then $\pair{\alg{L}}{F}$ is a model of $\logic{L}$, $F$ is a prime upset on $\alg{L}$, and the rule in question fails in $\pair{\alg{L}}{F}$.
\end{proof}

\begin{fact}
  Let $\logic{L}$ be a finitary extension of $\BD{1}$. Then $\logic{L}$ enjoys the $n$-PCP if and only if it is complete with respect to some set of structures of the form $\pair{\alg{L}}{F}$ where $\alg{L}$ is a De Morgan lattice and $F$ is an $n$-prime filter of $\alg{L}$.
\end{fact}

\begin{proof}
  The proof is entirely analogous.
\end{proof}

  We have therefore established the following theorems. Note that in the Boolean case each substructure of $\BAm{n}$ has the form $\BAm{m}$ and each substructure of $(\BAm{1})^{n}$ has the form $(\BAm{1})^{m}$ for some $m \leq n$, but this is far from true for $\DMm{n}$ and $(\DMm{1})^{n}$.

\begin{theorem}[Finitely generated extensions of $\BD{1}$] \label{thm: bd1 with n-pcp}
  The following are equivalent for each extension $\logic{L}$ of $\BD{1}$:
\begin{enumerate}[(i)]
\item $\logic{L}$ is a finitary and enjoys the $n$-PCP,
\item $\logic{L}$ is complete with respect to some set of substructures of $(\DMm{1})^{n}$,
\item $\logic{L}$ is complete with respect to some finite set of finite structures of the form $\pair{\alg{L}}{F}$ where $\alg{L}$ is a De~Morgan lattice and $F$ is an $n$-prime upset of $\alg{L}$.
\end{enumerate}
  Moreover, some such $n$ exists for each finitely generated extension of $\BD{1}$.
\end{theorem}

\begin{theorem}[Finitely generated extensions of $\BDinfty$ with the PCP] \label{thm: bdinfty with pcp}
  The following are equivalent for each extension $\logic{L}$ of $\BDinfty$:
\begin{enumerate}[(i)]
\item $\logic{L}$ is a finitary extension of $\BD{n}$ with the PCP,
\item $\logic{L}$ is complete with respect to some set of substructures of $\DMm{n}$,
\item $\logic{L}$ is complete with respect to some finite set of finite structures of the form $\pair{\alg{L}}{F}$ where $\alg{L}$ is a De~Morgan lattice and $F$ is a prime $n$-filter of $\alg{L}$.
\end{enumerate}
  Some such $n$ exists for each finitely generated extension of $\BDinfty$ with the PCP.
\end{theorem}

  It may well turn out that finitarity is not required here. However, proving this would require a finer analysis of the lattice of extensions of $\BDinfty$.

  These theorems organize the finitely generated extensions of $\BD{1}$ and the finitely generated extensions of $\BDinfty$ with the PCP into two hierarchies. We now explain how to extract a finite axiomatization for these logics.

  Consider a logic $\logic{L}$ determined by a finite set of finite structures $\class{K}$ which lies at the $n$-th level of one of these hierarchies. Then $\logic{L}$ is complete with respect to some set of substructures of a given finite structure, either $\DMm{n}$ or $(\DMm{1})^{n}$. If such a substructure $\pair{\alg{A}}{F}$ is a model of $\logic{L}$, there is a finitary semantic construction witnessing this (the Leibniz reduct of $\pair{\alg{A}}{F}$ is a strict homomorphic image of a substructure of some finite product of structures in $\class{K}$, per~\cite{dellunde+jansana96}). If, on the other hand, $\pair{\alg{A}}{F}$ is not a model of~$\logic{L}$, then there is a finitary rule $(\rho)$ which holds in $\logic{L}$ but not in $\pair{\alg{A}}{F}$. Applying this alternative to each of the finitely many substructures $\pair{\alg{A}}{F}$ will yield a finite set of finitary rules. The logic $\logic{L}$ is then the smallest logic at the given level of the hierarchy which satisfies these rules.

  The logic $\logic{L}$ is thus the smallest extension of $\BDinfty$ which enjoys the $n$-PCP and validates a certain finite set of finitary rules, or the smallest extension of $\BD{n}$ which enjoys the PCP and validates a certain finite set of finitary rules. This is essentially a Gentzen-style axiomatization of $\logic{L}$. In the latter case, it can further be transformed into a Hilbert-style axiomatization.

  Given a finitary rule $\Gamma \vdash \varphi$, we define its \emph{disjunctive variant} as the rule
\begin{align*}
  \set{\gamma \vee x}{\gamma \in \Gamma} \vdash \varphi \vee x, \text{ where $x$ does not occur in $\Gamma \vdash \varphi$}.
\end{align*}
  The following argument was already used by Font~\cite{font97} and Shramko~\cite{shramko19fde}.

\begin{fact}
  The smallest extension of $\BDinfty$ with the PCP which validates a set of finitary rules~$R$ is the extension of $\BDinfty$ by the disjunctive variants of the rules~$R$.
\end{fact}

\begin{proof}
  Since $\logic{L}$ has the PCP, the disjunctive variant of each rule in $R$ also holds in $\logic{L}$. Conversely, let $\logic{L}$ be the logic axiomatized by the disjunctive variants of the rules in $R$. By induction over the complexity of proof, we can show that if $\Gamma \vdash \varphi$ is valid in $\logic{L}$, then so is its disjunctive variant. Now suppose that $\Gamma, \varphi_{1} \vdash_{\logic{L}} \psi$ and $\Gamma, \varphi_{2} \vdash_{\logic{L}} \psi$. Then $\set{\gamma \vee \varphi_{2}}{\gamma \in \Gamma}, \varphi_{1} \vee \varphi_{2} \vdash_{\logic{L}} \psi \vee \varphi_{2} \vdash_{\logic{L}} \varphi_{2} \vee \psi$ and $\set{\gamma \vee \psi}{\gamma \in \Gamma}, \varphi_{2} \vee \psi \vdash_{\logic{L}} \psi \vee \psi \vdash_{\logic{L}} \psi$. But $\gamma \vdash_{\logic{L}} \gamma \vee \varphi_{2}$ and $\gamma \vdash_{\logic{L}} \gamma \vee \psi$, therefore $\Gamma, \varphi_{1} \vee \varphi_{2} \vdash_{\logic{L}} \varphi_{2} \vee \psi$ and $\Gamma, \varphi_{2} \vee \psi \vdash_{\logic{L}} \psi$. It follows that $\Gamma, \varphi_{1} \vee \varphi_{2} \vdash_{\logic{L}} \psi$.
\end{proof}

  We therefore obtain a constructive proof of the following theorem.

\begin{theorem}[Finite basis theorem for finitely generated logics of prime upsets]
  Each finitely generated extension of $\BDinfty$ with the PCP is has a finite Hilbert-style axiomatization.
\end{theorem}

\begin{theorem}[Finite basis theorem for finitely generated logics of filters]
  Each finitely generated extension $\logic{L}$ of $\BD{1}$ with the $n$-PCP is the smallest among extensions of $\BD{1}$ with the $n$-PCP which validate a certain finite set of finitary rules.
\end{theorem}

  In the other case, what we obtain is a Gentzen-style axiomatization. Given a set of finitary rules $R$, consider the sequent system over sequents of the form $\Gamma \seq \varphi$ where $\Gamma$ is a set of formulas and $\varphi$ is a formula with the following rules and axioms:
\begin{itemize}
\item axioms for all substitution instances of some rule in $R$,
\item the Identity axiom: $\varphi \seq \varphi$ for each $\varphi$,
\item the Weakening rule: from $\Gamma \seq \varphi$ infer $\Gamma, \Delta \seq \varphi$,
\item the finitary Cut rule: from $\Gamma \seq \varphi$ and $\varphi, \Delta \seq \psi$ infer $\Gamma, \Delta \seq \psi$,
\item the $n$-PCP rule:
\begin{align*}
  \text{from $\Gamma, \bigvee_{j \neq 1} \! \varphi_{j} \seq \psi$ and~\dots~and $\Gamma, \!\!\! \bigvee_{j \neq n+1} \!\!\!\! \varphi_{j} \seq \psi$ infer $\Gamma, \varphi_{1} \vee \dots \vee \varphi_{n+1} \seq \psi$.}
\end{align*}
\end{itemize}
  Let us call this sequent system the $n$-PCP sequent calculus with the axioms $R$.

\begin{fact} \label{fact: gentzen}
  Let $\logic{L}$ be the smallest extension of $\BDinfty$ with $n$-PCP which validates a set of finitary rules $R$. Then ${\Gamma \vdash_{\logic{L}} \varphi}$ if and only if $\Gamma \seq \varphi$ is provable in the $n$-PCP sequent calculus with the axioms~$R$.
\end{fact}

\begin{proof}
  It suffices to observe that the sequents provable in this calculus form a logic with the $n$-PCP which validates each rule in $R$, and conversely if the premises of a rule of this calculus are valid in $\logic{L}$, then so it the conclusion.
\end{proof}

\subsection{Finitary extensions of $\BD{1}$ with the $2$-PCP}
\label{subsec: bd 1 extensions}

  We now explicitly describe the first two levels in the hierarchy of finitely \mbox{generated} extensions of Belnap–Dunn logic $\BDplain \assign \BD{1}$, i.e.\ finitary extensions of $\BDplain$ with the $2$-PCP, or equivalently logics determined by the some family of substructures of $\DMm{1}$.

  There are two three-valued substructures which determine the strong Kleene logic~$\Kplain \assign \K{1}$ and the Logic of Paradox~$\LPplain \assign \LP{1}$, one two-valued substructure which determines classical logic $\CLplain \assign \CL{1}$, and two singleton substructures which determine the trivial logic $\TRIV$ and the almost trivial logic $\ATRIV$. In the trivial logic $\Gamma \vdash \varphi$ holds always, while in the almost trivial logic $\Gamma \vdash \varphi$ holds if and only if $\Gamma$ is non-empty. Intersecting this logic with any logic with theorems will result in a theoremless version of the logic. For example, $\CLplain \cap \ATRIV$ is the almost classical logic, where $\Gamma \vdash \varphi$ holds if and only if $\Gamma$ is non-empty and the rule holds classically. The first level of the hierarchy therefore consists of the intersections of the logics determined by these substructures, namely:
\begin{align*}
  \BDplain, \LPplain \cap \Kplain, \Kplain, \LPplain, \CLplain, \LPplain \cap \ATRIV, \CLplain \cap \ATRIV, \ATRIV, \TRIV.
\end{align*}
  The logic $\KOplain \assign \LPplain \cap \Kplain$ is known under various names, as discussed in footnote~\ref{foot: kalman}.

\begin{figure}
\caption{Proper substructures of $(\DMm{1})^{2}$}
\label{fig: sub1}
\vskip 15pt
\begin{center}
\begin{tikzpicture}[scale=0.75, dot/.style={inner sep=2.5pt,outer sep=2.5pt}, solid/.style={circle,fill,inner sep=2pt,outer sep=2pt}, empty/.style={circle,draw,inner sep=2pt,outer sep=2pt}]
  \node (f) at (0,0) [empty] {};
  \node (n) at (-1,1) [empty] {};
  \node (b) at (1,1) [solid] {};
  \node (t) at (0,2) [solid] {};
  \draw[-] (f) -- (n) -- (t);
  \draw[-] (f) -- (b) --(t);
  \node (a) at (0,-1) {$\DMm{1}$};
\end{tikzpicture}
\qquad
\begin{tikzpicture}[scale=0.75, dot/.style={inner sep=2.5pt,outer sep=2.5pt}, solid/.style={circle,fill,inner sep=2pt,outer sep=2pt}, empty/.style={circle,draw,inner sep=2pt,outer sep=2pt}]
  \node (f) at (0,0) [empty] {};
  \node (b) at (0,1) [solid] {};
  \node (t) at (0,2) [solid] {};
  \draw[-] (f) -- (b) -- (t);
  \node (a) at (0,-1) {$\Pm{1}$};
\end{tikzpicture}
\qquad
\begin{tikzpicture}[scale=0.75, dot/.style={inner sep=2.5pt,outer sep=2.5pt}, solid/.style={circle,fill,inner sep=2pt,outer sep=2pt}, empty/.style={circle,draw,inner sep=2pt,outer sep=2pt}]
  \node (f) at (0,0) [empty] {};
  \node (n) at (0,1) [empty] {};
  \node (t) at (0,2) [solid] {};
  \draw[-] (f) -- (n) -- (t);
  \node (a) at (0,-1) {$\Km{1}$};
\end{tikzpicture}
\qquad
\begin{tikzpicture}[scale=0.75, dot/.style={inner sep=2.5pt,outer sep=2.5pt}, solid/.style={circle,fill,inner sep=2pt,outer sep=2pt}, empty/.style={circle,draw,inner sep=2pt,outer sep=2pt}]
  \node (f) at (0,0) [empty] {};
  \node (t) at (0,2) [solid] {};
  \draw[-] (f) -- (t);
  \node (a) at (0,-1) {$\BAm{1}$};
\end{tikzpicture}
\qquad
\begin{tikzpicture}[scale=0.75, dot/.style={inner sep=2.5pt,outer sep=2.5pt}, solid/.style={circle,fill,inner sep=2pt,outer sep=2pt}, empty/.style={circle,draw,inner sep=2pt,outer sep=2pt}]
  \node (f) at (0,1) [empty] {};
  \node (a) at (0,-1) {$\Amatrix$};
\end{tikzpicture}
\end{center}
\vskip 10pt
\begin{center}
\begin{tikzpicture}[scale=0.75, dot/.style={inner sep=2.5pt,outer sep=2.5pt}, solid/.style={circle,fill,inner sep=2pt,outer sep=2pt}, empty/.style={circle,draw,inner sep=2pt,outer sep=2pt}]
  \node (0) at (0,0) [empty] {};
  \node (11) at (-1,1) [empty] {};
  \node (12) at (1,1) [empty] {};
  \node (21) at (-2,2) [empty] {};
  \node (22) at (0,2) [empty] {};
  \node (23) at (2,2) [empty] {};
  \node (31) at (-1,3) [empty] {};
  \node (32) at (1,3) [empty] {};
  \node (41) at (0,4) [solid] {};
  \draw[-] (0) -- (11) -- (21);
  \draw[-] (12) -- (22) -- (31);
  \draw[-] (23) -- (32) -- (41);
  \draw[-] (21) -- (31) -- (41);
  \draw[-] (11) -- (22) -- (32);
  \draw[-] (0) -- (12) -- (23);
  \node (a) at (0,-1) {$\Mnine$};
\end{tikzpicture}
\qquad
\begin{tikzpicture}[scale=0.75, dot/.style={inner sep=2.5pt,outer sep=2.5pt}, solid/.style={circle,fill,inner sep=2pt,outer sep=2pt}, empty/.style={circle,draw,inner sep=2pt,outer sep=2pt}]
  \node (0) at (0,0) [empty] {};
  \node (11) at (-1,1) [empty] {};
  \node (12) at (1,1) [empty] {};
  \node (21) at (-2,2) [empty] {};
  \node (22) at (0,2) [empty] {};
  \node (31) at (-1,3) [empty] {};
  \node (32) at (1,3) [empty] {};
  \node (41) at (0,4) [solid] {};
  \draw[-] (0) -- (11) -- (21);
  \draw[-] (12) -- (22) -- (31);
  \draw[-] (32) -- (41);
  \draw[-] (21) -- (31) -- (41);
  \draw[-] (11) -- (22) -- (32);
  \draw[-] (0) -- (12);
  \node (a) at (0,-1) {$\Meight$};
\end{tikzpicture}
\qquad
\begin{tikzpicture}[scale=0.75, dot/.style={inner sep=2.5pt,outer sep=2.5pt}, solid/.style={circle,fill,inner sep=2pt,outer sep=2pt}, empty/.style={circle,draw,inner sep=2pt,outer sep=2pt}]
  \node (0) at (0,0) [empty] {};
  \node (11) at (-1,1) [empty] {};
  \node (12) at (1,1) [empty] {};
  \node (22) at (0,2) [empty] {};
  \node (31) at (-1,3) [empty] {};
  \node (32) at (1,3) [empty] {};
  \node (41) at (0,4) [solid] {};
  \draw[-] (0) -- (11);
  \draw[-] (12) -- (22) -- (31);
  \draw[-] (32) -- (41);
  \draw[-] (31) -- (41);
  \draw[-] (11) -- (22) -- (32);
  \draw[-] (0) -- (12);
  \node (a) at (0,-1) {$\Mseven$};
\end{tikzpicture}
\qquad
\begin{tikzpicture}[scale=0.75, dot/.style={inner sep=2.5pt,outer sep=2.5pt}, solid/.style={circle,fill,inner sep=2pt,outer sep=2pt}, empty/.style={circle,draw,inner sep=2pt,outer sep=2pt}]
  \node (f) at (0,0) [empty] {};
  \node (n) at (-1,1) [empty] {};
  \node (b) at (1,1) [empty] {};
  \node (t) at (0,2) [solid] {};
  \draw[-] (f) -- (n) -- (t);
  \draw[-] (f) -- (b) --(t);
  \node (a) at (0,-1) {$\Mfour$};
\end{tikzpicture}
\end{center}
\vskip 10pt
\begin{center}
\begin{tikzpicture}[scale=0.75, dot/.style={inner sep=2.5pt,outer sep=2.5pt}, solid/.style={circle,fill,inner sep=2pt,outer sep=2pt}, empty/.style={circle,draw,inner sep=2pt,outer sep=2pt}]
  \node (0) at (0,0) [empty] {};
  \node (11) at (-1,1) [empty] {};
  \node (12) at (1,1) [empty] {};
  \node (21) at (-2,2) [solid] {};
  \node (22) at (0,2) [empty] {};
  \node (23) at (2,2) [empty] {};
  \node (31) at (-1,3) [solid] {};
  \node (32) at (1,3) [empty] {};
  \node (41) at (0,4) [solid] {};
  \draw[-] (0) -- (11) -- (21);
  \draw[-] (12) -- (22) -- (31);
  \draw[-] (23) -- (32) -- (41);
  \draw[-] (21) -- (31) -- (41);
  \draw[-] (11) -- (22) -- (32);
  \draw[-] (0) -- (12) -- (23);
  \node (a) at (0,-1) {$\Nnine$};
\end{tikzpicture}
\qquad
\begin{tikzpicture}[scale=0.75, dot/.style={inner sep=2.5pt,outer sep=2.5pt}, solid/.style={circle,fill,inner sep=2pt,outer sep=2pt}, empty/.style={circle,draw,inner sep=2pt,outer sep=2pt}]
  \node (0) at (0,0) [empty] {};
  \node (11) at (-1,1) [empty] {};
  \node (12) at (1,1) [empty] {};
  \node (21) at (-2,2) [solid] {};
  \node (22) at (0,2) [empty] {};
  \node (31) at (-1,3) [solid] {};
  \node (32) at (1,3) [empty] {};
  \node (41) at (0,4) [solid] {};
  \draw[-] (0) -- (11) -- (21);
  \draw[-] (12) -- (22) -- (31);
  \draw[-] (32) -- (41);
  \draw[-] (21) -- (31) -- (41);
  \draw[-] (11) -- (22) -- (32);
  \draw[-] (0) -- (12);
  \node (a) at (0,-1) {$\Neight$};
\end{tikzpicture}
\qquad
\begin{tikzpicture}[scale=0.75, dot/.style={inner sep=2.5pt,outer sep=2.5pt}, solid/.style={circle,fill,inner sep=2pt,outer sep=2pt}, empty/.style={circle,draw,inner sep=2pt,outer sep=2pt}]
  \node (0) at (0,0) [empty] {};
  \node (11) at (-1,1) [empty] {};
  \node (12) at (1,1) [empty] {};
  \node (22) at (0,2) [empty] {};
  \node (31) at (-1,3) [solid] {};
  \node (32) at (1,3) [empty] {};
  \node (41) at (0,4) [solid] {};
  \draw[-] (0) -- (11);
  \draw[-] (12) -- (22) -- (31);
  \draw[-] (32) -- (41);
  \draw[-] (31) -- (41);
  \draw[-] (11) -- (22) -- (32);
  \draw[-] (0) -- (12);
  \node (a) at (0,-1) {$\Nseven$};
\end{tikzpicture}
\end{center}
\end{figure}
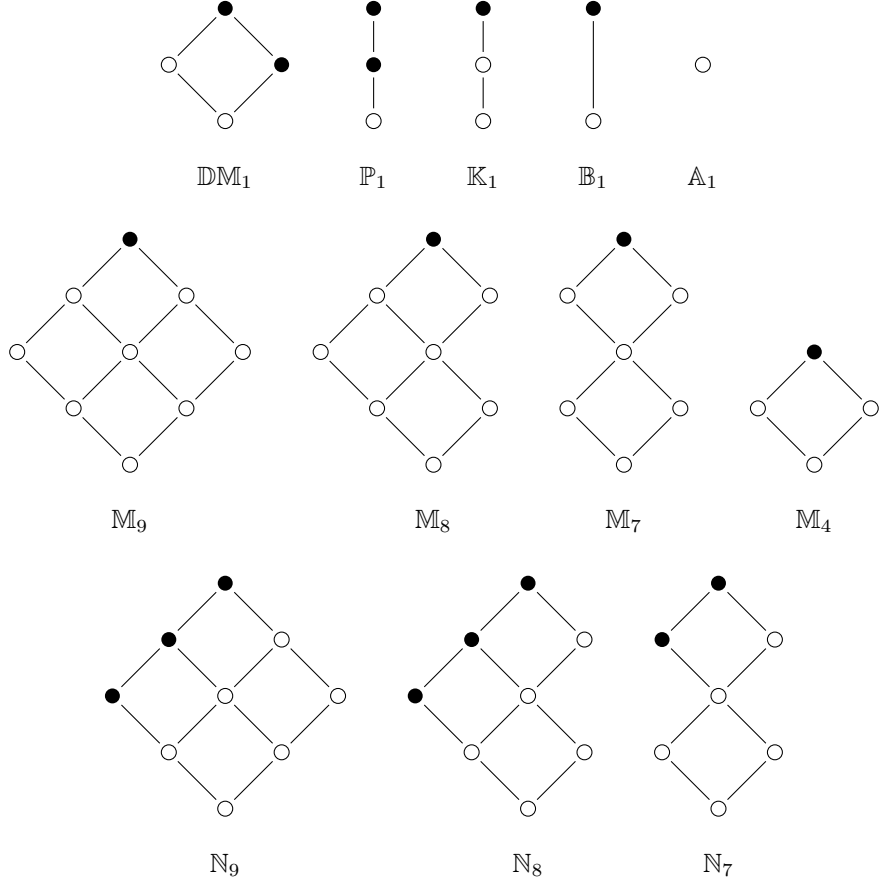

  (Expanding the signature of our logic by the constants $\True$ and $\False$ would have the effect of removing the singleton structures from the list of sub\-structures of $\DMm{1}$, or equivalently removing the almost trivial logic $\ATRIV$ and its intersections with $\LPplain$ and $\CLplain$ from the above list of logics.)

  The next level of the hierarchy consists of the finitary extensions of $\BDplain$ with the $2$-PCP, or equivalently of logics determined by some family of substructures of~$(\DMm{1})^{2}$. As luck would have it, all of these logics were already considered in~\cite{prenosil18} and their relative positions in the lattice of extensions of $\BDplain$ are known. In the following, we shall therefore make free use of facts established in~\cite{prenosil18}.

  In addition to the logics already contained in the first level, the second level of the hierarchy contains the Exactly True Logic~$\ETL$ of Pietz \& Rivieccio~\cite{pietz+rivieccio13}, which extends $\BDplain$ by the disjunctive syllogism $x, \neg x \vee y \vdash y$. This logic is complete with respect to the structure $\Mfour$. It also contains the logic $\ECQomega$, which extends $\BDplain$ by the infinite set of rules
\begin{align*}
  (x_{1} \wedge \neg x_{1}) \vee \dots \vee (x_{n} \wedge \neg x_{n}) \vdash y,
\end{align*}
  and the logic~$\Kminus$, which extends $\BDplain$ by the infinite set of rules
\begin{align*}
  (x_{1} \wedge \neg x_{1}) \vee \dots \vee (x_{n} \wedge \neg x_{n}) \vee y, \neg y \vee z \vdash z.
\end{align*}
  The logic $\ECQomega$ is complete with respect to the structure $\DMm{1} \times \BAm{1}$ and the logic $\Kminus$ is complete with respect to the structure $\Meight$ shown in Figure~\ref{fig: sub1}. (The logic $\Kminus$ is denoted in this way because it is a lower cover of $\Kplain$ in the lattice of extensions of $\BDplain$, while the notation $\ECQomega$ stands for \emph{ex contradictione quodlibet}.) We note that all $14$ extensions of $\LPplain \cap \Kminus$ are represented at this leve ofo the hierarchy, i.e.\ each extension of $\LPplain \cap \Kminus$ has the $2$-PCP.

\begin{theorem}[Finitary extensions of $\BD{1}$ with the $2$-PCP]
  The finitary extensions of $\BDplain$ which enjoy the $2$-PCP are precisely the logics shown in Figure~\ref{fig: bd1 2pcp}.
\end{theorem}

\begin{proof}
  Each such logic is determined by some set of substructures of~$(\DMm{1})^{2}$. Because the structure $(\DMm{1})^{2}$ itself determines the logic $\BDplain$, it suffices to consider its proper substructures. We first determine these up to isomorphism. Each sub\-structure of $(\DMm{1})^{2}$ is isomorphic to either a binary product of $\BAm{1}$, $\Km{1}$, $\Pm{1}$, $\DMm{1}$ or to one of the structures in Figure~\ref{fig: sub1}. The solid dots indicate the designated elements. In each structure the De~Morgan involution is the reflection across the horizontal axis of symmetry. (The duality theory of Cornish \& Fowler~\cite{cornish+fowler77} for De~Morgan algebras is helpful in determining these substructures.)

  Secondly, we observe that some of these structures are logically equivalent. Let $\structure{A} \leqhs \structure{B}$ abbreviate the claim that the structure $\structure{A}$ is a strict homomorphic image of a substructure of $\structure{B}$. In particular, $\structure{A} \leqhs \structure{B}$ implies that the logic determined by $\structure{A}$ extends the logic determined by $\structure{B}$.

  The structure $\Km{1}$ is a strict homo\-morphic image of $\Nseven$, so $\Nseven$ and $\Km{1}$ are logically equivalent. In addition, $\DMm{1} \leqhs \Neight \leqhs \Nnine \leqhs (\DMm{1})^{2}$, so $\Neight$ and $\Nnine$ determine the logic $\BDplain$. We therefore do not need to consider the structures $\Nseven$--$\Nnine$. Moreover, $\Km{1} \leqhs \Mseven \leqhs \Km{1} \times \Km{1}$ and $\Mfour \leqhs \Mnine \leqhs \Mfour \times \Mfour$, so $\Mseven$ and $\Km{1}$ are logically equivalent, as are $\Mnine$ and $\Mfour$. We therefore only need to consider $\Mfour$, $\Meight$, and the binary products of $\BAm{1}$, $\Km{1}$, $\Pm{1}$, $\DMm{1}$.

  Thirdly, we identify the logics determined by these structures. This was already done in~\cite{prenosil18}: in addition to the logics determined by substructures of $\DMm{1}$, we obtain the logics $\ETL$ (for $\Mfour$), $\Kminus$ (for $\Meight$), $\ECQomega$ (for $\BAm{1} \times \DMm{1}$, as well as for $\Km{1} \times \DMm{1}$), $\LPplain \cup \ECQomega$ (for $\BAm{1} \times \Pm{1}$), $\KOplain \cup \ECQomega$ (for $\Km{1} \times \Pm{1}$). The structure $\Pm{1} \times \DMm{1}$ determines $\BDplain$. Figure~\ref{fig: sub1} now shows all intersections of these logics, as determined in~\cite{prenosil18}.
\end{proof}

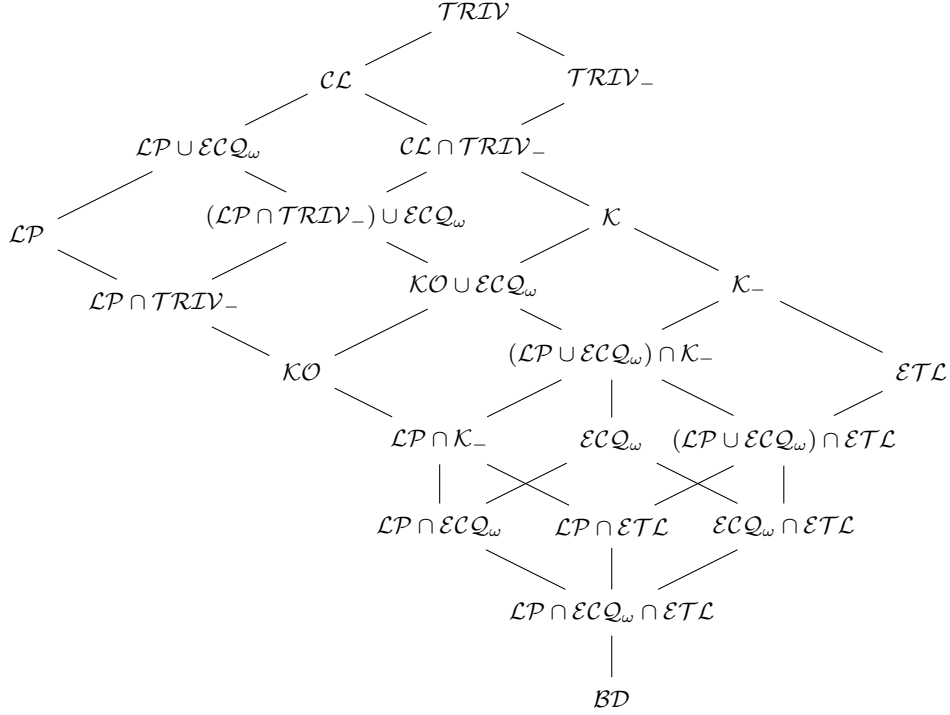
\begin{figure}
\caption{Extensions of $\BDplain$ with the $2$-PCP}
\label{fig: bd1 2pcp}
\vskip 15pt
\begin{center}
\begin{tikzpicture}[scale=0.7625, dot/.style={inner sep=2.5pt,outer sep=2.5pt}]
  \node (BD) at (0, -0.5) [dot] {\small $\BDplain$};
  \node (atom) at (0, 1) [dot] {\small $\LPplain \cap \ECQomega \cap \ETL$};
  \node (LPECQ) at (-3, 2.5) [dot] {\small $\LPplain \cap \ECQomega$};
  \node (LPETL) at (0, 2.5) [dot] {\small $\LPplain \cap \ETL$};
  \node (ECQETL) at (3, 2.5) [dot] {\small $\ECQomega \cap \ETL$};
  \node (LPKm) at (-3, 4) [dot] {\small $\LPplain \cap \Kminus$};
  \node (ECQ) at (0, 4) [dot] {\small $\ECQomega$};
  \node (LPECQKm) at (0, 5.5) [dot] {\small $(\LPplain \cup \ECQomega) \cap \Kminus$};
  \node (LPECQETL) at (3, 4) [dot] {\small $(\LPplain \cup \ECQomega) \cap \ETL$};
  \node (ETL) at (5.4, 5.2) [dot] {\small $\ETL$};
  \node (Km) at (2.4, 6.7) [dot] {\small $\Kminus$}; 
  \node (KO) at (-5.4, 5.2) [dot] {\small $\KOplain$};
  \node (KOECQ) at (-2.4, 6.7) [dot] {\small $\KOplain \cup \ECQomega$};
  \node (K) at (0, 7.9) [dot] {\small $\Kplain$};
  \node (LP) at (-10.2, 7.6) [dot] {\small $\LPplain$};
  \node (LPcupECQ) at (-7.2, 9.1) [dot] {\small $\LPplain \cup \ECQomega$};
  \node (CL) at (-4.8, 10.3) [dot] {\small $\CLplain$};
  \node (TRIV) at (-2.4, 11.5) [dot] {\small $\TRIV$};
  \node (ALP) at (-7.8, 6.4) [dot] {\small $\LPplain \cap \ATRIV$};
  \node (ALPcupECQ) at (-4.8, 7.9) [dot] {\small $(\LPplain \cap \ATRIV) \cup \ECQomega$};
  \node (ACL) at (-2.4, 9.1) [dot] {\small $\CLplain \cap \ATRIV$};
  \node (ATRIV) at (0, 10.3) [dot] {\small $\ATRIV$};
  \draw[-] (BD) -- (atom) edge (LPECQ) edge (LPETL) edge (ECQETL);
  \draw[-] (LPETL) edge (LPKm) edge (LPECQETL);
  \draw[-] (ECQ) edge (LPECQ) edge (ECQETL) edge (LPECQKm);
  \draw[-] (LPECQETL) edge (LPECQKm) edge (ETL) edge (ECQETL);
  \draw[-] (LPKm) edge (KO) edge (LPECQKm) edge (LPECQ);
  \draw[-] (Km) edge (ETL) edge (LPECQKm) edge (K);
  \draw[-] (KOECQ) edge (KO) edge (LPECQKm) edge (K) edge (ALPcupECQ);
  \draw[-] (CL) edge (LPcupECQ) edge (ACL) edge (TRIV);
  \draw[-] (ALP) edge (KO) edge (ALPcupECQ);
  \draw[-] (LP) edge (ALP) edge (LPcupECQ);
  \draw[-] (LPcupECQ) edge (ALPcupECQ);
  \draw[-] (ACL) edge (ALPcupECQ) edge (K) edge (ATRIV);
  \draw[-] (ATRIV) edge (TRIV);
\end{tikzpicture}
\end{center}
\end{figure}

  It is worth noting that the lattice of all extensions of $\BDplain$ with the $2$-PCP is distributive. We do not know whether this is a coincidence or part of a pattern.

  Several of the logics in Figure~\ref{fig: bd1 2pcp} are not finitely axiomatizable by a Hilbert-style calculus. Nonetheless, as described above, they are all complete with respect to a finite sequent calculus. For example, to obtain a sequent calculus for $\ECQomega$ it suffices to add the sequent axiom schema $(x_{1} \wedge \neg x_{1}) \vee (x_{2} \wedge \neg x_{2}) \seq y$ to the $2$-PCP sequent calculus for $\BDplain$, because $\ECQomega$ is the smallest extension of $\BDplain$ with the $2$-PCP which validates this rule. Replacing this axiom by $x \wedge \neg x \seq y$ yields a sequent calculus for $\ECQomega \cap \ETL$, and replacing it by $(x_{1} \wedge \neg x_{1}) \vee y, \neg y \vee z \seq z$ yields a sequent calculus for $\Kminus$.

\subsection{Finitary extensions of $\BDinfty$ with the PCP}
\label{subsec: bd infty with pcp}

  Similarly, we can describe the finitely generated extensions of $\BDinfty$ with the PCP. The first level of this hierarchy of finitely generated extensions of $\BDinfty$ with the PCP, consisting of the finitary extensions of $\BD{1}$ with the PCP, was already described in the previous subsection. The second level consists of logics the finitary extensions of $\BD{2}$ with the PCP, or equivalently logics determined by some family of substructures of~$\DMm{2}$.

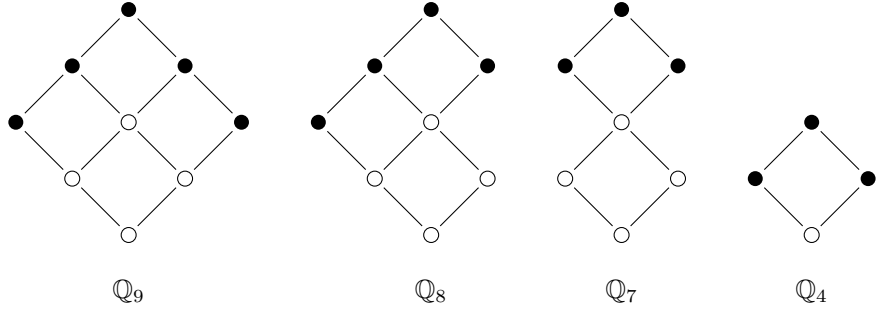
\begin{figure}
\caption{Some of the substructures of $\DMm{2}$}
\label{fig: sub2}
\vskip 15pt
\begin{center}
\begin{tikzpicture}[scale=0.75, dot/.style={inner sep=2.5pt,outer sep=2.5pt}, solid/.style={circle,fill,inner sep=2pt,outer sep=2pt}, empty/.style={circle,draw,inner sep=2pt,outer sep=2pt}]
  \node (0) at (0,0) [empty] {};
  \node (11) at (-1,1) [empty] {};
  \node (12) at (1,1) [empty] {};
  \node (21) at (-2,2) [solid] {};
  \node (22) at (0,2) [empty] {};
  \node (23) at (2,2) [solid] {};
  \node (31) at (-1,3) [solid] {};
  \node (32) at (1,3) [solid] {};
  \node (41) at (0,4) [solid] {};
  \draw[-] (0) -- (11) -- (21);
  \draw[-] (12) -- (22) -- (31);
  \draw[-] (23) -- (32) -- (41);
  \draw[-] (21) -- (31) -- (41);
  \draw[-] (11) -- (22) -- (32);
  \draw[-] (0) -- (12) -- (23);
  \node (a) at (0,-1) {$\Qnine$};
\end{tikzpicture}
\qquad
\begin{tikzpicture}[scale=0.75, dot/.style={inner sep=2.5pt,outer sep=2.5pt}, solid/.style={circle,fill,inner sep=2pt,outer sep=2pt}, empty/.style={circle,draw,inner sep=2pt,outer sep=2pt}]
  \node (0) at (0,0) [empty] {};
  \node (11) at (-1,1) [empty] {};
  \node (12) at (1,1) [empty] {};
  \node (21) at (-2,2) [solid] {};
  \node (22) at (0,2) [empty] {};
  \node (31) at (-1,3) [solid] {};
  \node (32) at (1,3) [solid] {};
  \node (41) at (0,4) [solid] {};
  \draw[-] (0) -- (11) -- (21);
  \draw[-] (12) -- (22) -- (31);
  \draw[-] (32) -- (41);
  \draw[-] (21) -- (31) -- (41);
  \draw[-] (11) -- (22) -- (32);
  \draw[-] (0) -- (12);
  \node (a) at (0,-1) {$\Qeight$};
\end{tikzpicture}
\qquad
\begin{tikzpicture}[scale=0.75, dot/.style={inner sep=2.5pt,outer sep=2.5pt}, solid/.style={circle,fill,inner sep=2pt,outer sep=2pt}, empty/.style={circle,draw,inner sep=2pt,outer sep=2pt}]
  \node (0) at (0,0) [empty] {};
  \node (11) at (-1,1) [empty] {};
  \node (12) at (1,1) [empty] {};
  \node (22) at (0,2) [empty] {};
  \node (31) at (-1,3) [solid] {};
  \node (32) at (1,3) [solid] {};
  \node (41) at (0,4) [solid] {};
  \draw[-] (0) -- (11);
  \draw[-] (12) -- (22) -- (31);
  \draw[-] (32) -- (41);
  \draw[-] (31) -- (41);
  \draw[-] (11) -- (22) -- (32);
  \draw[-] (0) -- (12);
  \node (a) at (0,-1) {$\Qseven$};
\end{tikzpicture}
\qquad
\begin{tikzpicture}[scale=0.75, dot/.style={inner sep=2.5pt,outer sep=2.5pt}, solid/.style={circle,fill,inner sep=2pt,outer sep=2pt}, empty/.style={circle,draw,inner sep=2pt,outer sep=2pt}]
  \node (f) at (0,0) [empty] {};
  \node (n) at (-1,1) [solid] {};
  \node (b) at (1,1) [solid] {};
  \node (t) at (0,2) [solid] {};
  \draw[-] (f) -- (n) -- (t);
  \draw[-] (f) -- (b) --(t);
  \node (a) at (0,-1) {$\Qfour$};
\end{tikzpicture}
\end{center}
\end{figure}

\begin{theorem}[Finitary extensions of $\BD{2}$ with the PCP]
  The lattice of finitary extensions of $\BD{2}$ with the PCP is iso\-morphic to the lattice of downsets of the poset shown in in Figure~\ref{fig: si dmas}.
\end{theorem}

\begin{proof}
  We again start by determining these substructures up to logical equivalence. The structure $\DMm{2}$ itself determines the logic $\BD{2}$. The sub\-structures (or sets of substructures) where the designated elements form a filter will yield one of the logics $\BDplain$, $\KOplain$, $\Kplain$, $\LPplain$, $\CLplain$, $\ATRIV$, or $\TRIV$. These logics are generated by $\DMm{1}$, $\{ \Pm{1}, \Km{1} \}$, $\Km{1}$, $\Pm{1}$, $\BAm{1}$, and $\Amatrix$.

  We therefore only need to describe the proper substructures of $\DMm{2}$ where the designated elements do not form a filter. These come in two types: the binary \emph{dual} products of the structures $\BAm{1}$, $\Km{1}$, $\Pm{1}$, $\DMm{1}$, and the four substructures of the structure $\Qnine$ shown in of Figure~\ref{fig: sub2}.

  Secondly, we determine the $\HsS$-order on these substructures, where ${\structure{M} \leqhs \structure{N}}$ abbreviates the claim that $\structure{M}$ is a strict homomorphic image of a sub\-structure of~$\structure{N}$. This yields the diagram in Figure~\ref{fig: si dmas}. More precisely, at this point we only need to check that if $\structure{M}$ lies below $\structure{N}$ in this diagram, then $\structure{M} \leqhs \structure{N}$.

  Thirdly, we show that for each structure $\structure{M}$ there is a rule $(\rho)$, which we call a \emph{separating rule} for $\structure{M}$, such that $(\rho)$ fails in $\structure{M}$ but it holds in every structure in Figure~\ref{fig: si dmas} which does not lie above $\structure{M}$. Thus two distinct families of structures from Figure~\ref{fig: si dmas} which are downward closed in the given order determine distinct logics, since they can be distinguished by one of the separating rules. To axiomatize the logic of such a downward closed family of substructures of $\DMm{2}$ relative to $\BD{2}$, it therefore suffices to throw in the disjunctive variant of the separating rule for each of the (minimal) structures in Figure~\ref{fig: si dmas} which lie outside of this family.

  The separating rules for the structures in Figure~\ref{fig: si dmas} are listed in Table~\ref{tab: separating rules}. Verifying that these are indeed separating rules is a time-consuming but mechanical affair. Readers who wish to do this on their own may find the following two observations helpful: if $\Gamma \vdash \varphi$ fails in $\structure{M}$ and $\Delta \vdash \varphi$ fails in $\structure{N}$, then $\Gamma, \Delta \vdash \varphi$ fails in $\structure{M} \otimes \structure{N}$. On the other hand, if $\Gamma \vdash \varphi$ holds in~$\structure{M}$ and no $\gamma \in \Gamma$ can be designated in $\structure{N}$, then $\Gamma \vdash \varphi$ holds in~$\structure{M} \otimes \structure{N}$.
\end{proof}

  We do not know whether the fact that the lattice of extensions of $\BD{2}$ with the PCP turns out to be distributive is a coincidence or not.

\begin{figure}
\caption{The $\HsS$-order on the substructures of $\DMm{2}$}
\label{fig: si dmas}
\vskip 15pt
\begin{center}
\begin{tikzpicture}[scale=1, dot/.style={inner sep=2.5pt,outer sep=2.5pt}]
  \node (B2) at (0,0) [dot] {$\BAm{1}$};
  \node (B2B2) at (0,1) [dot] {$\BAm{1} \otimes \BAm{1}$};
  \node (K3) at (-3.5,1) [dot] {$\Km{1}$};
  \node (P3) at (3.5,1) [dot] {$\Pm{1}$};
  \node (B2K3) at (-3.5,3) [dot] {$\BAm{1} \otimes \Km{1}$};
  \node (B2P3) at (3.5,3) [dot] {$\BAm{1} \otimes \Pm{1}$};
  \node (DM4) at (0,2) [dot] {$\DMm{1}$};
  \node (B2DM4) at (0,5) [dot] {$\BAm{1} \otimes \DMm{1}$};
  \node (K3P3) at (0,6) [dot] {$\Km{1} \otimes \Pm{1}$};
  \node (K3K3) at (-3.5,5) [dot] {$\Km{1} \otimes \Km{1}$};
  \node (P3P3) at (3.5,5) [dot] {$\Pm{1} \otimes \Pm{1}$};
  \node (K3DM4) at (-3.5,7) [dot] {$\Km{1} \otimes \DMm{1}$};
  \node (P3DM4) at (3.5,7) [dot] {$\Pm{1} \otimes \DMm{1}$};
  \node (DM4DM4) at (0,8) [dot] {$\DMm{1} \otimes \DMm{1}$};
  \node (M4) at (1.166, 3) [dot] {$\Qfour$};
  \node (M7) at (-1.166,2.333) [dot] {$\Qseven$};
  \node (M8) at (0, 3) [dot] {$\Qeight$};
  \node (M9) at (0, 4) [dot] {$\Qnine$};
  \node (A1) at (-3.5, 0) [dot] {$\Amatrix$};
  \draw[-] (B2) -- (K3) -- (B2K3) edge (B2DM4) edge (K3P3) -- (K3K3) -- (K3DM4) -- (DM4DM4);
  \draw[-] (B2) -- (P3) -- (B2P3) edge (B2DM4) edge (K3P3) -- (P3P3) -- (P3DM4) -- (DM4DM4);
  \draw[-] (B2) -- (B2B2) edge (B2K3) edge (B2P3);
  \draw[-] (K3P3) edge (K3DM4) edge (P3DM4);
  \draw[-] (A1) -- (K3) -- (DM4);
  \draw[-] (P3) edge (DM4) edge (M4);
  \draw[-] (DM4) -- (K3DM4);
  \draw[-] (B2DM4) edge (K3DM4) edge (P3DM4);
  \draw[-] (DM4) -- (M8);
  \draw[-] (M4) edge (M9);
  \draw[-] (M8) edge (M9);
  \draw[-] (M9) edge[bend right=60] (DM4DM4);
  \draw[-] (K3) -- (M7) -- (M8);
  \draw[-] (M7) edge (K3K3);
\end{tikzpicture}
\end{center}
\end{figure}
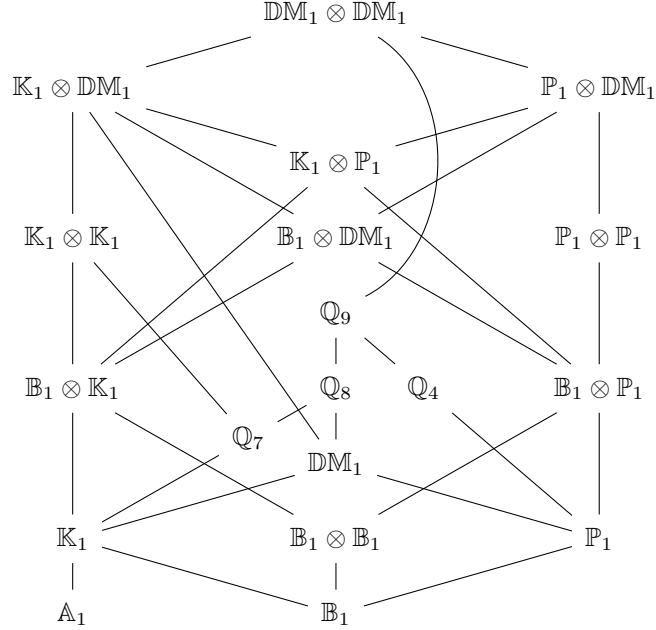

\begin{table}
\caption{Separating rules for substructures of $\DMm{2}$}
\label{tab: separating rules}
\begin{framed}
\begin{align*}
  & \BAm{1}: x \vdash y \\
  & \Amatrix: \emptyset \vdash x \vee \neg x \\
  & \Km{1}: x \vdash y \vee \neg y \\
  & \Pm{1}: x \wedge \neg x \vdash y \\
  & \DMm{1}: x \wedge \neg x \vdash y \vee \neg y \\
  & \Qfour: x \wedge \neg x, y \wedge \neg y \vdash (x \vee \neg x) \wedge (y \vee \neg y) \\
  & \Qseven: x, y \vdash \neg x \vee \neg y \vee (x \wedge y) \\
  & \Qeight: x \wedge \neg x, y \vdash (x \wedge y) \vee \neg y \\
  & \Qnine: x \wedge \neg x \wedge z, y \wedge \neg y \vdash (x \vee \neg x \vee z) \wedge (y \vee \neg y \vee \neg z) \\
  & \BAm{1} \otimes \BAm{1}: x, \neg x \vdash x \wedge \neg x \\
  & \BAm{1} \otimes \Pm{1}: x \wedge \neg x, y, \neg y \vdash y \wedge \neg y \\
  & \BAm{1} \otimes \Km{1}: x, \neg x, y \vee \neg y\vdash x \wedge (\neg x \vee y \vee \neg y) \\
  & \BAm{1} \otimes \DMm{1}: x \wedge y, x \vee \neg x, y \wedge \neg y \vdash (x \vee \neg x) \wedge y \wedge \neg y \\
  & \Km{1} \otimes \Km{1}: x \wedge z, y \wedge \neg z \vdash (x \wedge y) \vee \neg (x \wedge y) \\
  & \Km{1} \otimes \Pm{1}: x, \neg x \wedge y \wedge \neg y \vdash x \wedge (y \vee \neg y) \\
  & \Km{1} \otimes \DMm{1}: x \wedge z, y \wedge \neg y \wedge \neg z \vdash (x \vee \neg x) \wedge (y \vee \neg y) \\
  & \Pm{1} \otimes \Pm{1}: x \wedge y, x \wedge \neg x, y \wedge \neg y \vdash x \wedge \neg x \wedge y \wedge \neg y \\
  & \Pm{1} \otimes \DMm{1}: x \wedge \neg x \wedge y, y \wedge \neg y \wedge z \wedge \neg z \vdash (x \vee \neg x) \wedge z \\
  & \DMm{1} \otimes \DMm{1}: x \wedge \neg x \wedge z \wedge \neg z, y \wedge \neg y \wedge z \wedge \neg z \vdash (x \vee \neg x) \wedge (y \vee \neg y)
\end{align*}
\end{framed}
\end{table}

    Our list of separating rules provides us with a finite Hilbert-style axiomatization of each extension of $\BD{2}$ with the~PCP. Consider for example the logic determined by the structure $\pair{\DM{1}}{\{ \True, \Neither, \Both \}}$ (labelled $\Qfour$ in Figure~\ref{fig: sub2}). The logic determined by this structure is the logic of ``anything but falsehood'' introduced by Shramko~\cite{shramko19abf} as the dual of the logic of ``exact truth'' introduced by Pietz \& Rivieccio~\cite{pietz+rivieccio13}. More precisely, Shramko \mbox{considers} the \textsc{Fmla}-\textsc{Set} consequence relation of this structure and axiomatizes it in his paper, while we wish to study its \textsc{Set}-\textsc{Fmla} consequence relation. (Shramko's axiomatization is essentially dual to the axiomatization of the \textsc{Set}-\textsc{Fmla} fragment of Exactly True Logic. However, such a duality cannot be applied to obtain a \textsc{Set}-\textsc{Fmla} axiomatization.)

\begin{theorem}[Completeness for Shramko's logic of anything but falsehood] \label{thm: completeness for abf}
  The logic of the structure $\pair{\DM{1}}{\{ \True, \Neither, \Both \}}$ is the extension of $\BDinfty$ by the $2$-adjunction rule, the law of the excluded middle $\emptyset \vdash x \vee \neg x$, and the rule $x \vee y, \neg x \vee y \vdash (x \wedge \neg x) \vee y$.
\end{theorem}

\begin{proof}
  The minimal structures in the poset in Figure~\ref{fig: si dmas} which do not lie below $\Qfour$ are $\Amatrix$ and $\BAm{1} \otimes \BAm{1}$. Their separating rules are $\emptyset \vdash x \vee \neg x$ and $x, \neg x \vdash x \wedge \neg x$. Consequently, the logic of $\Qfour$ is the smallest extension of $\CL{2}$ with the PCP which validates these two rules. But the disjunctive variants of these rules are precisely $\emptyset \vdash x \vee \neg x$ and $x \vee y, \neg x \vee y \vdash (x \wedge \neg x) \vee y$.
\end{proof}

\section{Metalogical properties of logics of upsets of De~Morgan lattices}

  In this final section, we make a few rudimentary observations concerning the classification of logics of upsets of De~Morgan lattices within the Leibniz and Frege hierarchies of abstract algebraic logic (see the textbook~\cite{font16} for an introduction to these hierarchies). This will partly extend the existing classification of super-Belnap logics, i.e.\ extensions of $\BD{1}$~\cite{albuquerque+prenosil+rivieccio17,prenosil22}.

  Two prominent classes in the Leibniz hierarchy are the classes of protoalgebraic and truth-equational logics. A logic $\logic{L}$ is \emph{protoalgebraic} if it has a \emph{protoimplication set}, i.e.\ a set of formulas in two variables $\Delta(x, y)$ such that
\begin{align*}
  \emptyset \vdash_{\logic{L}} \delta(x, x) \text{ for each } \delta(x, y) \in \Delta(x, y), & & x, \Delta(x, y) \vdash_{\logic{L}} y.
\end{align*}

\begin{theorem}[Protoalgebraic extensions of $\BDinfty$] \label{thm: protoalgebraic}
  The only protoalgebraic extension of $\BDinfty$ other than the trivial and almost trivial logic is classical logic $\CL{1}$.
\end{theorem}

\begin{proof}
  Let $\logic{L}$ be a protoalgebraic extension of $\BDinfty$ other than the trivial or almost trivial logic. Then $\logic{L} \leq \CL{1}$ and $\logic{L}$ has a protoimplication set $\Delta(x, y)$. Up to logical equivalence (interderivability), $\BDinfty$ and its extensions only have finitely many formulas. Thus there is, up to logical equivalence in $\BDinfty$, a logically strongest set of formulas $\Delta(x, y)$ such that $x, \Delta(x, y) \vdash_{\CL{1}} y$, namely
\begin{align*}
  \Delta(x, y) \assign \{ (\neg x \vee y) \wedge (\neg x \vee x) \wedge (\neg y \vee y) \}.
\end{align*}
  On the other hand, if $\logic{L}$ has a theorem, then one such theorem must be $\neg x \vee x$. Thus
\begin{align*}
  & x, (\neg x \vee y) \wedge (\neg x \vee x) \wedge (\neg y \vee y) \vdash_{\logic{L}} y, & & \emptyset \vdash_{\logic{L}} \neg x \vee x.
\end{align*}
  Substituting $x \wedge y$ for $y$ yields $x, (\neg x \vee x) \wedge (\neg x \vee y) \wedge (\neg x \vee \neg y \vee x) \wedge (\neg x \vee \neg y \vee y) \vdash_{\logic{L}} x \wedge y$, therefore
\begin{align*}
  x, \neg x \vee (x \wedge y) \vdash_{\logic{L}} x \wedge y.
\end{align*}
  Substituting $\neg x \vee x$ for $x$ in this rule yields $\neg x \vee x, (x \wedge \neg x) \vee (x \wedge y) \vee (\neg x \wedge y) \vdash_{\logic{L}} (\neg x \vee x) \wedge y$, hence $(x \wedge \neg x) \vee ((\neg x \vee x) \wedge y) \vdash_{\logic{L}} (\neg x \vee x) \wedge y$. Further substituting $\neg y \vee y$ for $y$ yields $(x \wedge \neg x) \vee ((\neg x \vee x) \wedge (\neg y \vee y)) \vdash_{\logic{L}} (\neg x \vee x) \wedge (\neg y \vee y)$. But $\emptyset \vdash_{\logic{L}} (x \wedge \neg x) \vee ((x \vee \neg x) \wedge (y \vee \neg y))$. This is because, taking $t(x, y)$ to be the right-hand side of this rule, the inequality $\neg t(x, y) \leq t(x, y)$ holds in all De~Morgan lattices, hence $t(x, y) = \neg t(x, y) \vee t(x, y)$ in all De~Morgan lattices. But we know that $\emptyset \vdash_{\logic{L}} \neg x \vee x$, in particular $\emptyset \vdash_{\logic{L}} \neg t(x,y) \vee t(x, y)$. We thus obtain the rule
\begin{align*}
  \emptyset \vdash_{\logic{L}} (\neg x \vee x) \wedge (\neg y \vee y).
\end{align*}
  It follows that $\emptyset \vdash_{\logic{L}} \neg x \vee (x \wedge (y \vee \neg y))$. Combining this with the substitution instance $x, \neg x \vee (x \wedge (\neg y \vee y) \vdash_{\logic{L}} x \wedge (\neg y \vee y)$ of the rule $x, \neg x \vee (x \wedge y) \vdash_{\logic{L}} (x \wedge y)$ yields $x \vdash_{\logic{L}} x \wedge (\neg y \vee y)$, hence $y \vdash_{\logic{L}} \neg x \vee (x \wedge y)$. Combining this again with the rule $x, \neg x \vee (x \wedge y) \vdash_{\logic{L}} x \wedge y$ yields the adjunction rule $x, y \vdash_{\logic{L}} x \wedge y$. Thus $\logic{L}$ is an extension of $\BD{1}$. But it was proved in~\cite{albuquerque+prenosil+rivieccio17} that, with the two stated exceptions, the only protoalgebraic extension of $\BD{1}$ is $\CL{1}$. (The simplest proof of this relies on the observation that, with the two exceptions, if $\logic{L} \nleq \CL{1}$, then $\logic{L} \leq \K{1}$ or $\logic{L} \leq \LP{1} \vee \ECQomega$, but neither of these two logics is protoalgebraic.)
\end{proof}

  The above theorem remains valid if we expand $\BDinfty$ by a constant $\True$ representing the top element and restrict to non-empty upsets. (In that case, we only need to further justify that $\emptyset \vdash_{\logic{L}} x \vee \neg x$. This involves showing that if $\emptyset \vdash_{\logic{L}} \Delta (x, x)$ but $\emptyset \nvdash_{\logic{L}} x \vee \neg x$, then $\Delta(x, y)$ must be equivalent in $\BDinfty$ to $\True$. But $x, \True \nvdash_{\logic{L}} y$.)

  The definition of a truth-equational logic involves the \emph{Leibniz congruence} of a structure $\pair{\alg{A}}{F}$, denoted $\Omega^{\alg{A}} F$, which is the largest congruence $\theta$ on $\alg{A}$ compatible with $F$ in the sense that $a \in F$ and $\pair{a}{b} \in \theta$ implies $b \in F$. A logic $\logic{L}$ is called \emph{truth-equational} if there is a set of equations in one variable $\Eps(x)$ such that for each model $\pair{\alg{A}}{F}$ of $\logic{L}$ we have
\begin{align*}
  a \in F & \iff \pair{\delta(a)}{\epsilon(a)} \in \Omega^{\alg{A}} F \text{ for each equation } \delta(x) \approx \epsilon(x) \in \Eps(x). 
\end{align*}
  That is, ``truth'' in models of $\logic{L}$ (the set $F$) is defined in terms of ``equations'' (pairs of elements identified by the Leibniz congruence of $F$).

\begin{lemma}[Leibniz relations of upsets of De~Morgan lattices]
  Let $F$ be an upset of a De~Morgan lattice $\alg{L}$. Then $\pair{a}{b} \in \Omega^{\alg{A}} F$ if and only if for each $c, d \in \alg{L}$:
\begin{align*}
  (a \wedge c) \vee d \in F & \iff (b \wedge c) \vee d \in F, \\
  (\neg a \wedge c) \vee d \in F & \iff (\neg b \wedge c) \vee d \in F.
\end{align*}
\end{lemma}

\begin{proof}
  These conditions must be satisfied by each congruence compatible with $F$. Conversely, these conditions define a congruence which is compatible with $F$.
\end{proof}

  If a logic $\logic{L}$ is truth-equational with a set of defining equations $\Eps(x)$, then the same holds for every extension of $\logic{L}$. In particular, excluding the trivial and almost trivial logic, if an extension of $\BDinfty$ is truth-equational with a set of defining equations $\Eps(x)$, then so is classical logic. This constrains the possible sets of equations $\Eps(x)$. Indeed, the reader can verify that the only possible set of equations, up to equivalence in all De~Morgan lattices, is $\Eps(x) \assign \{ x \vee \neg x \approx x \}$. If we include the constant for the top element in our signature, this allows for the further equation $x \approx 1$. Logics truth-equational with respect to this defining equation are precisely the extensions of the Exactly True logic determined by $\pair{\DM{1}}{ \{ \True \} }$.

\begin{theorem}[Truth-equational extensions of $\BDinfty$]
  An extension $\logic{L}$ of $\BDinfty$ is truth-equational with defining equation $x \vee \neg x \approx x$ if and only if it validates the rules $\emptyset \vdash x \vee \neg x$ and $x, ((x \vee \neg x) \wedge y) \vee z \vdash (x \wedge y) \vee z$.
\end{theorem}

\begin{proof}
  If $\logic{L}$ is truth-equational with the given defining equation and $\pair{\alg{A}}{F}$ is a model of $\logic{L}$, then $a \vee \neg a \in F$ for each $a \in \alg{A}$, since $(a \vee \neg a) \vee \neg (a \vee \neg a) = a \vee \neg a$. Moreover, if $a \in F$, then $\pair{a \vee \neg a}{a} \in \Omega^{\alg{A}} F$, hence $((a \vee \neg a) \wedge b) \vee c \in F$ implies $(a \wedge b) \vee c \in F$. The two given rules are thus valid. Conversely, if the second rule is valid, then $a \in F$ implies $\pair{a \vee \neg a}{a} \in \Omega^{\alg{A}} F$ for each $a \in \alg{A}$, since the converse rule $(x \wedge y) \vee z \vdash ((x \vee \neg x) \wedge y) \vee z$ holds in each extension of $\BDinfty$. On the other hand, if $\pair{a \vee \neg a}{a} \in \Omega^{\alg{A}} F$ and the first rule is valid, then $a \vee \neg a \in F$, hence $a \in F$.
\end{proof}

  Other prominent classes of logics in the Leibniz hierarchy, such as equivalential or algebraizable logics, are stronger than protoalgebraicity, therefore this concludes our classification of extensions of $\BDinfty$ in the Leibniz hierarchy.

  The basic class in the Frege hierarchy is the class of selfextensional logics. A logic $\logic{L}$ is called \emph{selfextensional} if the interderivability relation $\varphi \dashv \vdash_{\logic{L}} \psi$ is a congruence on the algebra of formulas. That is, $\varphi \dashv \vdash_{\logic{L}} \psi$ implies $\chi(\varphi) \dashv \vdash_{\logic{L}} \chi(\psi)$, where $\chi(\varphi)$ is the result of uniformly substituting $\varphi$ for one of the variables of $\chi$. Equivalently, interderivability on $\logic{L}$ coincides with equational validity in some class of algebras. In our case, an extension $\logic{L}$ of $\BDinfty$ (other than the trivial and almost trivial logic) is selfextensional if and only if interderivability in $\logic{L}$ coincides with equational validity in De~Morgan lattices, or in Kleene lattices, or in Boolean lattices, since these are the only non-trivial varieties of De~Morgan lattices.

  In particular, one can observe that all logics in the three intervals $[\BDinfty, \BD{1}]$, $[\KOinfty, \KO{1}]$, and $[\CLinfty, \CL{1}]$ are selfextensional, since they share the same inter\-derivability relation. However, unlike protoalgebraic and truth-equational extensions, the selfextensional extensions of $\BDinfty$ seem difficult to pin down precisely. We content ourselves with observing that not every such extension falls under one of the three obvious cases mentioned above.

\begin{fact}
  The extension of $\BDinfty$ by the rule $x \wedge \neg x \vdash y \vee \neg y$ is selfextensional.
\end{fact}

\begin{proof}
  In this extension $\logic{L}$, we can show that $\Gamma \vdash_{\logic{L}} \varphi$ if and only if for some $\gamma \in \Gamma$ either $\gamma \vdash_{\BDinfty} \varphi$ or there are $\alpha$ and $\beta$ such that $\gamma \vdash_{\BDinfty} \alpha \wedge \neg \alpha$ and $\beta \vee \neg \beta \vdash_{\BDinfty} \varphi$. (It suffices to check that this equivalence indeed defines a logic.) Thus if $\varphi \dashv \vdash_{\logic{L}} \psi$, then $\varphi \dashv \vdash_{\BDinfty} \psi$, since otherwise without loss of generality $\varphi \vdash_{\BDinfty} \alpha \wedge \neg \alpha$ and $\beta \vee \neg \beta \vdash_{\logic{L}} \psi$ and either $\psi \vdash_{\BDinfty} \varphi$ or there are $\alpha'$ and $\beta'$ such that $\psi \vdash_{\BDinfty} \alpha' \wedge \neg \alpha'$ and $\beta' \vee \neg \beta' \vdash_{\BDinfty} \psi$. In either case we have formulas $\gamma$ and $\delta$ such that $\gamma \vee \neg \gamma \vdash_{\BDinfty} \delta \wedge \neg \delta$. But then the inequalities $\gamma \leq \delta \leq \gamma$ and $\neg \gamma \leq \delta \leq \neg \gamma$ hold in all De~Morgan lattices, which is impossible: the equality $\gamma \approx \neg \gamma$ cannot be satisfied in the two-element Boolean algebra.
\end{proof}

  Finally, we identify some splittings of the lattice of extensions of $\BDinfty$, with the aim of providing constraints on logics where conjunction and disjunction satisfy certain minimal conditions.

\begin{fact}
  Let $\logic{L} \geq \BDinfty$. Then $\logic{L} \leq \BD{1}$ or $x \wedge \neg x \vdash_{\logic{L}} y \vee \neg y$, but not both.
\end{fact}

\begin{proof}
  The rule $x \wedge \neg x \vdash y \vee \neg y$ fails in $\BD{1}$, so the two disjuncts cannot both hold. If $x \wedge \neg x \nvdash_{\logic{L}} y \vee \neg y$, then this rule fails in some model of $\logic{L}$ of the form $\pair{\alg{A}}{F}$ where $\alg{A}$ is a De~Morgan lattice and $F$ is an upset of $\alg{A}$. Restricting to the appropriate substructure of this model, we may take $\alg{A}$ to be $2$-generated. It follows that the model $\pair{\alg{A}}{F}$ of $\logic{L}$ is a strict homomorphic image of some model $\pair{\FreeDMLat{x, y}}{G}$ of $\logic{L}$, where $\FreeDMLat{x, y}$ is the free De~Morgan lattice generated by $\{ x, y \}$ and $G$ is an upset of this algebra such that $x \wedge \neg x \in G$ but $y \vee \neg y \notin G$. Restricting again to the substructure generated by $x \wedge \neg x$ and $y \vee \neg y$, we obtain a model $\pair{\alg{B}}{H}$ of $\logic{L}$ generated by $x \wedge \neg x \in H$ and $y \vee \neg y \notin H$.  For each $b \in \alg{B}$ we have either $x \wedge \neg x \leq b$ or $b \leq y \vee \neg y$, therefore the upset $H$ is uniquely determined by $x \wedge \neg x \in H$ and $y \vee \neg y \notin H$ (see Figure~5 of~\cite{prenosil21}). The structure $\DMm{1}$ is then a strict homomorphic image of $\pair{\alg{B}}{H}$. Therefore $\DMm{1}$ is a model of $\logic{L}$ and $\logic{L} \leq \BD{1}$.
\end{proof}

\begin{fact}
  Let $\logic{L} \geq \BDinfty$. Then $\logic{L} \leq \ATRIV$ or $\emptyset \vdash_{\logic{L}} x \vee \neg x$, but not both.
\end{fact}

\begin{proof}
  The rule $\emptyset \vdash_{\logic{L}} x \vee \neg x$ fails in $\ATRIV$, so the two disjuncts cannot both hold. If $\emptyset \nvdash_{\logic{L}} x \vee \neg x$, then the two-element Boolean chain with no designated elements is a model of $\logic{L}$, in which case $\logic{L} \leq \ATRIV$.
\end{proof}

\begin{fact}
  Let $\logic{L} \geq \BDinfty$. Then $\logic{L} \leq \K{1}$ or $x \vdash_{\logic{L}} y \vee \neg y$, but not both.
\end{fact}

\begin{proof}
  The rule $x \vdash y \vee \neg y$ fails in $\K{1}$, so the two disjuncts cannot both hold.
\end{proof}

\begin{fact}
  Let $\logic{L} \geq \BDinfty$. Then $\logic{L} \leq \LP{1}$ or $x \wedge \neg x \vdash_{\logic{L}} y$, but not both. If $x \nvdash_{\logic{L}} y \vee \neg y$, then $\Km{1}$ is a model of $\logic{L}$, so $\logic{L} \leq \K{1}$.
\end{fact}

\begin{proof}
  The rule $x \wedge \neg x \vdash y$ fails in $\LP{1}$, so the two disjuncts cannot both hold. If $x \wedge \neg x \nvdash_{\logic{L}} y$, then $\Pm{1}$ is a model of $\logic{L}$, so $\logic{L} \leq \LP{1}$.
\end{proof}

  Let us say that an extension $\logic{L}$ of $\BDinfty$ satisfies the \emph{weak proof by cases property (weak PCP)} if
\begin{align*}
  \varphi_{1} \vee \varphi_{2} \vdash_{\logic{L}} \psi \text{ if } \varphi_{1} \vdash_{\logic{L}} \psi \text{ and } \varphi_{2} \vdash_{\logic{L}} \psi.
\end{align*}
  It satisfies the \emph{weak conjunction property (weak CP)} if
\begin{align*}
  \varphi \vdash_{\logic{L}} \psi_{1} \wedge \psi_{2} \text{ if } \varphi \vdash_{\logic{L}} \psi_{1} \text{ and } \varphi \vdash_{\logic{L}} \psi_{2}.
\end{align*}

\begin{theorem}[Extensions of $\BDinfty$ with the weak CP and PCP]
  An extension of $\BDinfty$ satisfies both the weak CP and the weak PCP if and only if it lies in one of the disjoint intervals $[\BDinfty, \BD{1}]$, $[\KOinfty, \KO{1}]$, $[\Kinfty, \K{1}]$, $[\LPinfty \cap \ATRIV, \LP{1}]$, $[\CLinfty, \CL{1}]$, or $\{ \ATRIV, \TRIV \}$.
\end{theorem}

\begin{proof}
  Let $\logic{L}$ be an extension of $\BDinfty$. If $\logic{L} \nleq \BD{1}$, then $x \wedge \neg x \vdash_{\logic{L}} y \vee \neg y$, so $((x \wedge \neg x) \wedge z) \vee u \vdash_{\logic{L}} ((y \vee \neg y) \wedge z) \vee u$ and $\KOinfty \leq \logic{L}$. Suppose therefore that $\logic{L}$ is an extension of $\KOinfty$. If $\logic{L} \nleq \KO{1} = \LP{1} \cap \K{1}$, then either $\logic{L} \nleq \LP{1}$ or $\logic{L} \nleq \K{1}$. In the former case, $x \wedge \neg x \vdash_{\logic{L}} y$, so $(x \wedge x) \vee y \vdash_{\logic{L}} y$ and $\Kinfty \leq \logic{L}$. In the latter case, $x \vdash_{\logic{L}} y \vee \neg y$, so $\LPinfty \cap \ATRIV \leq \logic{L}$. If $\Kinfty \leq \logic{L} \nleq \K{1}$, then $x \vdash_{\logic{L}} y \vee \neg y$, so $\CLinfty \leq \logic{L}$. Likewise, if $\LPinfty \cap \ATRIV \leq \logic{L} \nleq \LP{1}$, then $x \wedge \neg x \vdash_{\logic{L}} y$, so $\CLinfty \leq \logic{L}$.
\end{proof}

\begin{theorem}[Selfextensional extensions of $\BDinfty$ with the weak CP or PCP]
  An extension of $\BDinfty$ which satisfies either the weak CP or the weak PCP is self\-extensional if and only if it lies in one of the disjoint intervals $[\BDinfty, \BD{1}]$, $[\KOinfty, \KO{1}]$, $[\CLinfty \cap \ATRIV, \CL{1}]$, or $\{ \ATRIV, \TRIV \}$.
\end{theorem}

\begin{proof}
  The \textsc{Fmla}--\textsc{Fmla} fragment of a logic with the weak CP which validates the rule $x \wedge y \vdash x$ is uniquely determined by its interderivability relation, since it validates $\varphi \vdash \psi$ if and only if it validates $\varphi \dashv \vdash \varphi \wedge \psi$. The same holds for logics with the weak PCP which validate the rule $x \vdash x \vee y$. The interderivability relation in a selfextensional extension $\logic{L}$ of $\BDinfty$ coincides with equational validity in some variety of De~Morgan lattices, i.e.\ in De~Morgan lattices, Kleene lattices, Boolean algebras, or in the trivial variety. Equivalently, interderivability in $\logic{L}$ coincides with interderivability in $\BDinfty$, $\KOinfty$, $\CLinfty$, or $\ATRIV$. Because these four logics also satisfy the weak CP (or the weak PCP), the \textsc{Fmla}--\textsc{Fmla} fragment of $\logic{L}$ coincides with the \textsc{Fmla}--\textsc{Fmla} fragment of one of these four logics. If $\logic{L} \nleq \BD{1}$, then $x \wedge \neg x \vdash_{\logic{L}} y \vee \neg y$, hence $\logic{L}$ must extend the \textsc{Fmla}--\textsc{Fmla} fragment of $\KOinfty$. Given that $\KOinfty$ is axiomatized by its \textsc{Fmla}--\textsc{Fmla} fragment, this simply means that $\KOinfty \leq \logic{L}$. If $\logic{L} \nleq \KO{1}$, then either $\logic{L} \nleq \K{1}$ or $\logic{L} \nleq \LP{1}$, hence either $x \vdash_{\logic{L}} y \vee \neg y$ or $x \wedge \neg x \vdash_{\logic{L}} y$. In either case this means that $\logic{L}$ must extend the \textsc{Fmla}--\textsc{Fmla} fragment of $\CLinfty \cap \ATRIV$, which again simply means that $\CLinfty \cap \ATRIV \leq \logic{L}$. Finally, if $\logic{L} \nleq \CL{1}$, then $\ATRIV \leq \logic{L}$, hence $\logic{L} = \ATRIV$ or $\logic{L} = \TRIV$.
\end{proof}

%\bibliographystyle{plain}
%\bibliography{Logics_of_upsets}

\begin{thebibliography}{10}

\bibitem{albuquerque+prenosil+rivieccio17}
Hugo Albuquerque, Adam P\v{r}enosil, and Umberto Rivieccio.
\newblock An algebraic view of super-{B}elnap logics.
\newblock {\em Studia Logica}, (105):1051--1086, 2017.

\bibitem{belnap77a}
Nuel~D. Belnap.
\newblock How a computer should think.
\newblock In {\em Contemporary Aspects of Philosophy}, pages 30--56. Oriel
  Press Ltd., 1977.

\bibitem{belnap77b}
Nuel~D. Belnap.
\newblock A useful four-valued logic.
\newblock In J.~Michael Dunn and George Epstein, editors, {\em Modern uses of
  multiple-valued logic}, volume~2 of {\em Episteme}, pages 5--37. Springer
  Netherlands, 1977.

\bibitem{cintula+noguera13}
Petr Cintula and Carles Noguera.
\newblock The proof by cases property and its variants in structural
  consequence relations.
\newblock {\em Studia Logica}, 101:713--747, 2013.

\bibitem{cornish+fowler77}
William~H. Cornish and Peter~R. Fowler.
\newblock Coproducts of {De} {Morgan} algebras.
\newblock {\em Bulletin of the Australian Mathematical Society}, 16:1--13,
  1977.

\bibitem{czelakowski01}
Janusz Czelakowski.
\newblock {\em Protoalgebraic Logics}, volume~10 of {\em Trends in Logic:
  Studia Logica Library}.
\newblock Kluwer Academic Publishers, Dordrecht, 2001.

\bibitem{dellunde+jansana96}
Pilar Dellunde and Ramon Jansana.
\newblock Some characterization theorems for infinitary universal {Horn} logic
  without equality.
\newblock {\em Journal of Symbolic Logic}, 61(4):1242--1260, 1996.

\bibitem{dunn76}
J.~Michael Dunn.
\newblock Intuitive semantics for first-degree entailments and `coupled trees'.
\newblock {\em Philosophical Studies}, 29(3):149--168, 1976.

\bibitem{dunn99}
J.~Michael Dunn.
\newblock A comparative study of various model-theoretic treatments of
  negation: a history of formal negation.
\newblock In Heinrich Wansing and Dov~M. Gabbay, editors, {\em What is
  negation?}, pages 23--51. Kluwer Academic
  Publishers, 1999.

\bibitem{font16}
Josep~M. Font.
\newblock {\em Abstract algebraic logic -- an introductory textbook}, volume~60
  of {\em Studies in Logic: Mathematical Logic and Foundations}.
\newblock College Publications, 2016.

\bibitem{font97}
Josep~Maria Font.
\newblock Belnap's four-valued logic and {De Morgan} lattices.
\newblock {\em Logic Journal of the IGPL}, 5:1--29, 1997.

\bibitem{jaskowski69}
Stanis{\l}aw Ja\'{s}kowski.
\newblock Propositional calculus for contradictory deductive systems.
\newblock {\em Studia Logica}, 24:143--157, 1969.

\bibitem{jaskowski99}
Stanis{\l}aw Ja\'{s}kowski.
\newblock A propositional calculus for inconsistent deductive systems.
\newblock {\em Logic and Logical Philosophy}, 7:35--56, 1999.

\bibitem{kalman58}
John~Arnold Kalman.
\newblock Lattices with involution.
\newblock {\em Transactions of the American Mathematical Society},
  87(2):485--491, 1958.

\bibitem{makinson73}
David~Clement Makinson.
\newblock {\em Topics in Modern Logic}.
\newblock Methuen, London, 1973.

\bibitem{pietz+rivieccio13}
Andreas Pietz and Umberto Rivieccio.
\newblock Nothing but the {T}ruth.
\newblock {\em Journal of Philosophical Logic}, (42):125--135, 2013.

\bibitem{prenosil22}
Adam P\v{r}enosil.
\newblock Filter classes of upsets of distributive lattices.
\newblock Unpublished manuscript.

\bibitem{prenosil18}
Adam P\v{r}enosil.
\newblock {\em Reasoning with Inconsistent Information}.
\newblock PhD thesis, Charles University, 2018.

\bibitem{prenosil21}
Adam P\v{r}enosil.
\newblock The lattice of super-{Belnap} logics.
\newblock {\em The Review of Symbolic Logic}, 2021.

\bibitem{pynko99}
Alexej~P. Pynko.
\newblock Implicational classes of {De} {Morgan} lattices.
\newblock {\em Discrete Mathematics}, 205:171--181, 1999.

\bibitem{rivieccio12}
Umberto Rivieccio.
\newblock An infinity of super-{B}elnap logics.
\newblock {\em Journal of Applied Non-Classical Logics}, (22):319--335, 2012.

\bibitem{shramko19abf}
Yaroslav Shramko.
\newblock Dual-{Belnap} logic and anything but falsehood.
\newblock {\em Journal of Applied Logics -- IfCoLoG Journal of Logics and their
  Applications}, 6(2):413--430, 2019.

\bibitem{shramko19fde}
Yaroslav Shramko.
\newblock First-degree entailment and structural reasoning.
\newblock In Hitoshi Omori and Heinrich Wansing, editors, {\em New Essays on
  Belnap-Dunn logic}, volume 418 of {\em Synthese Library}. Springer, 2019.

\bibitem{shramko20}
Yaroslav Shramko.
\newblock Hilbert-style axiomatization of first-degree entailment and a family
  of its extensions.
\newblock {\em Annals of Pure and Applied Logic}, 172, 2020.

\end{thebibliography}

\end{document}